\documentclass[preprint,3p,11pt]{article}

\usepackage{graphicx}
\usepackage{subcaption}
\usepackage{color}
\usepackage{mwe}
\usepackage{wrapfig}
\usepackage{authblk}

\usepackage{epsf}

\usepackage{float} 
\usepackage{fullpage}
\usepackage{adjustbox}
\usepackage{amsfonts}
\usepackage{amsmath}
\usepackage{mathtools}
\usepackage{multicol}
\usepackage{amssymb}
\usepackage{amsbsy}
\usepackage{amsthm}
\usepackage{epsfig}
\usepackage{graphicx}

\usepackage{colordvi}
\usepackage{graphics}
\usepackage{color}
\usepackage{bm}
\usepackage{url}
\usepackage{algorithm}
\usepackage{algorithmic}
\usepackage{verbatim} 
\usepackage{braket}

\def\alex#1{
{\color{black}#1}
}

\def\nic#1{
{\color{black}#1}
}

\def\eli#1{
{\color{black}#1}
}

\usepackage{accents}

\newtheorem{theorem}{Theorem}[section]

\newcommand{\parphir}{{\phi_r}}
\newcommand{\parphid}{{\phi_d}}
\newcommand{\parphie}{{\phi_e}}
\newcommand{\parbetai}{{\beta_i}}
\newcommand{\parbetae}{{\beta_e}}
\newtheorem{remark}{Remark}

\begin{document}
\title{Delay differential equations for the spatially-resolved simulation of epidemics with specific application to COVID-19}
\author[1]{Nicola Guglielmi}
\author[2]{Elisa Iacomini}
\author[1]{Alex Viguerie}

\affil[1]{Gran Sasso Science Institute, Viale F. Crispi 7, L`Aquila, AQ 67100, Italy}
\affil[2]{ Institut für Geometrie und Praktische Mathematik (IGPM), RWTH Aachen University\\
	
	 Templergraben 55, 52062 Aachen, Germany}

\date{ }
\maketitle

\begin{abstract}
In the wake of the 2020 COVID-19 epidemic, much work has been performed on the development of mathematical models for the simulation of the epidemic, and of disease models generally. Most works follow the \textit{susceptible-infected-removed} (SIR) compartmental framework, modeling the epidemic with a system of ordinary differential equations. Alternative formulations using a partial differential equation (PDE) to incorporate both spatial and temporal resolution have also been introduced, with their numerical results showing potentially powerful descriptive and predictive capacity. In the present work, we introduce a new variation to such models by using delay differential equations (DDEs). The dynamics of many infectious diseases, including COVID-19, exhibit delays due to incubation periods and related phenomena. Accordingly, DDE models allow for a natural representation of the problem dynamics, in addition to offering advantages in terms of computational time and modeling, as they eliminate the need for additional, difficult-to-estimate, compartments (such as exposed individuals) to incorporate time delays. \alex{In the present work,} we introduce a DDE epidemic model in both an ordinary- and partial differential equation framework. \alex{We present a series of mathematical results assessing the stability of the formulation. We then perform \eli{several} numerical experiments, validating both the mathematical results and establishing model's ability to reproduce measured data on realistic problems. } \end{abstract}
\section{Introduction}\label{sec:intro}
The worldwide outbreak of COVID-19 in 2020 has caused unprecedented disruption, leading to massive damage in terms of both economic cost and human lives. Much of the economic damage in particular has been due to government efforts designed to retard the spread of the disease; while undoubtedly effective, the cost of such measures is enormous. In recent months, much research has focused on the mathematical modeling of the epidemic, and of epidemics generally, in the hope that such models may ultimately prove useful to decision-makers, and help to inform more targeted, less-disruptive interventions.
\par Many modeling approaches have been proposed, with some combining differential equations and empirical approaches in order to evaluate the effectiveness of various social-distancing measures \cite{Ferguson2020, Gatto202004978, GBB2020, LRGK2020}. In order to incorporate spatial variation across different regions, many of these models discretize various regions along geopolitical (or similar) lines, using a network structure to represent movement between the populations in different areas \cite{GBB2020, Gatto202004978, LRGK2020}. 
\par In contrast to these approaches, in \cite{VLABHPRYV2020, VLABHPRYV2020Due, jha2020bayesian,grave2020adaptive, BP2021, grave2021assessing}, the authors instead modeled the spatial diffusion of the disease using partial differential equation (PDE) models. The implemented models followed the compartmental framework, but also incorporated nonlinear heterogeneous diffusion terms, giving a reaction-diffusion system of equations. While the computational cost of such an approach is much higher than an ODE model, different numerical experiments showed that this approach carries several advantages. Notably, the timing of different dynamics across different areas can be resolved in a continuous manner, offering a richer description of the spatiotemporal evolution.
\par In the following, we propose an alternative formulation of the model introduced in \cite{VLABHPRYV2020} and analyzed further and extended in \cite{VLABHPRYV2020Due,jha2020bayesian,grave2020adaptive, BP2021}. In particular, we seek to model many of the dynamics, and in particular the incubation period, with a \textit{delay differential equation} model. Ordinary delay differential equation models have been extensively used for the study of epidemics, as well other types of biological models, such as predator-prey equations \cite{L2015,Kaddar2009OnTD,Bernoussi2014GlobalSO,forde2005delay,huang2011global,greenhalgh2017time}. Delayed models using partial differential equation (PDE) to study epidemics have been discussed and analyzed in e.g. \cite{smith2000global,pei2017traveling,yang2011travelling,wang2012traveling,ding2013traveling,zhao2018traveling}. Many of these works are restricted to mathematical analysis, though simple numerical tests were also carried out in \cite{zhao2018traveling,pei2017traveling}. A delay PDE simulation of a realistic problem over a nontrivial geometry has not, to the authors' knowledge, been carried out. 
\par There are several advantages in using a delay formulation rather than the system shown in \cite{VLABHPRYV2020}. Notably, the \textit{exposed} compartment, responsible for the incubation period, may be elimated without losing the relevant dynamics. For a PDE model in which problem size is relevant, obtaining the same dynamics with fewer compartments is desirable. However, we do not expect that the dynamics are exactly the same; we believe in fact that the delay-equation formulation may better capture the ``lag" effect incurred by the introduction of new measures (as seen during the COVID-19 pandemic), in which there is a delay of several days between the introduction of a new public health ordinance and when its effects are fist observed. These lags may also change depending on the epidemic stage, leading to \textit{state-dependent delays}. Though we will not examine such a case here, a thorough understanding of the constant-delay case is necessary first and is the objective of the present work. Such a formulation is also interesting from the mathematical and computational point of view, and examining such a  model is worthwhile, we believe, in and of itself.
\par This paper is outlined as follows. We begin by recalling the model shown in \cite{VLABHPRYV2020}, along with some of its basic properties and notation. We will then proceed to introduce the delay-equation formulation of the model for both ordinary and partial differential equation variants, highlighting important points of difference. Following this introduction, we will present several mathematical results of the delay-differential equation models including the equilibria solutions and a stability analysis. We will then perform a series of numerical examples using the ODE and an idealized 1D problem for the PDE (inspired by \cite{VLABHPRYV2020Due, grave2020adaptive}) to qualitatively analyze the model behavior and confirm the mathematical results. We finish our numerical tests with a simulation over the Italian region of Lombardy using real data, in order to validate the model's ability to reproduce real-life data on realistic problems, before concluding with several suggestions for future research in this area.

\section{Model}

\label{model}
The COVID-19 PDE model presented in \cite{VLABHPRYV2020} and further analyzed and extended in \cite{BP2021, VLABHPRYV2020Due, jha2020bayesian,grave2020adaptive, grave2021assessing} reads:
\begin{align}
\label{eq1} \partial_t s &= \alpha n - \left(1-A/n\right)\beta_i s i - \left(1-A/n\right)\beta_e s e - \mu s  + \nabla\cdot\left(n\, \nu_s \nabla s\right) \\
\label{eq2} \partial_t e &= \left(1-A/n\right)\beta_i si + \left(1-A/n\right)\beta_e se - \sigma e - \phi_e e - \mu e + \nabla\cdot\left(n\, \nu_e \nabla e \right) \\
\label{eq3} \partial_t i &= \sigma e - \phi_d \,i - \phi_r i - \mu i  + \nabla \cdot\left(n\, \nu_i \nabla i \right)\\
\label{eq4} \partial_t r &= \phi_r i + \phi_e e - \mu r  + \nabla \cdot\left(n\, \nu_r \nabla r \right) \\
\label{eq5} \partial_t d &= \phi_d\, i,
\end{align}
while a similar ODE variant, neglecting diffusion, reads:
\begin{align}
\label{eq1ODE} \dot{s} &= \alpha n - \beta_i s i - \beta_e s e - \mu s   \\
\label{eq2ODE} \dot{e} &= \beta_i si + \beta_e se - \sigma e - \phi_e e - \mu e \\
\label{eq3ODE} \dot{i} &= \sigma e - \phi_d \,i - \phi_r i - \mu i  \\
\label{eq4ODE} \dot{r} &= \phi_r i + \phi_e e - \mu r  \\
\label{eq5ODE} \dot{d} &= \phi_d\, i.
\end{align}

\begin{figure}\centering
  \includegraphics[width=\textwidth]{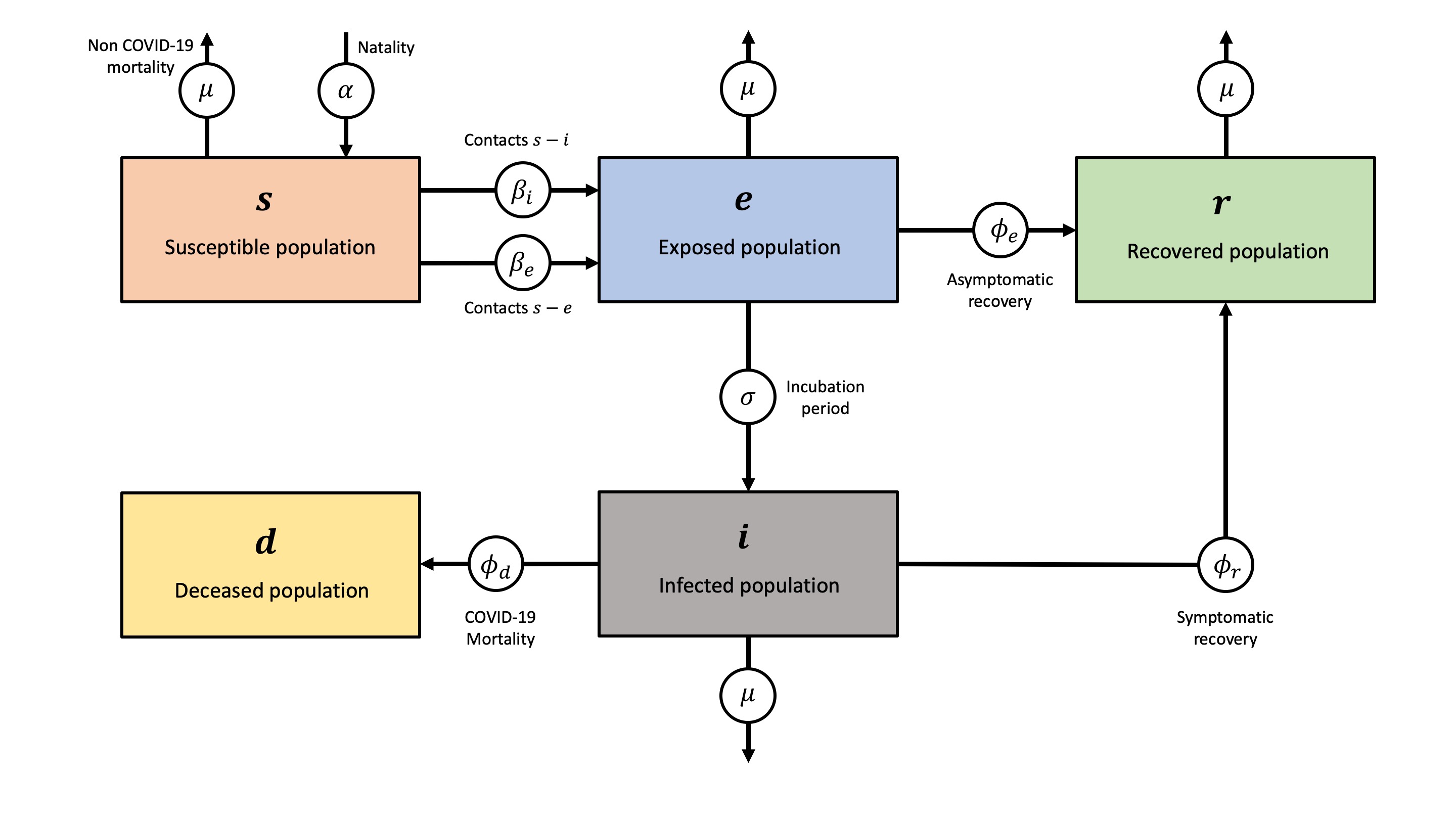}
 \caption{Flow chart describing the evolution of the various compartments and parameters in the model equations (\ref{eq1})-(\ref{eq5}).}\label{fig:flowChart}
\end{figure}

The mechanism of the model is diagrammed in Fig. \ref{fig:flowChart} and operates in the following way: the susceptible population $s$ is exposed to the disease by contact with exposed individuals in compartment $e$ or infected patients in compartment $i$ at rates $\beta_e$ and $\beta_i$, respectively. After an incubation period $\sigma$, exposed individuals develop symptoms and move to the infected subgroup $i$. A fraction of symptomatic patients recover at a rate $\phi_r$, moving into the recovered subgroup $r$. However, the remaining infected patients eventually die at a rate $\phi_d$. The model also features asymptomatic transmission, which has been considered a key driving force in the COVID-19 pandemic. To this end, we include a fraction of the exposed population $e$ that directly moves to the recovered subgroup $r$, without ever entering in the symptomatic infected compartment $i$. Note that:
\begin{align}\label{nDefn}
n=s+e+i+r	
\end{align}

\noindent denotes the entire living population.
\par We briefly make a few additional remarks regarding this model. The first is that this model operates on the principle of mass-action, with the contact terms $\beta_{i,e}$ non-normalized and hence dependent on local population densities, reflected in their units of 1/(Time$\cdot$Persons) \cite{VLABHPRYV2020Due, grave2020adaptive,MurrayI,MurrayII, grave2021assessing}. The spatial dependence of contagion is further augmented by the addition of the Allee term $A$, which accounts for the tendency of COVID-19 cases to cluster in areas where $n>>A$. This term has been used extensively in other settings, with the form used above inspired directly by applications in cancer modeling \cite{johnson2019cancer, delitala2020Cancer}. The Allee term works to reduce transmission in areas where the population density is under a given threshold $A$, by bringing the population in the exposed compartment to the susceptible compartment. Consequently, in such areas, the population in compartments $e$ and $i$ tends to zero, eventually cancelling out the transfer term. We observe that as $s$, $i$, and $e$ are all less than $n$ by definition, we do not expect blowup of this term, even for very small $n$.\alex{We note that the Allee term does not appear in the ODE model \eqref{eq1ODE}-\eqref{eq5ODE}. While one may technically include it, the lack of spatial variation in population density limits its usefulness from the modeling point of view, and conceptually its inclusion only makes sense in the PDE model \eqref{eq1}-\eqref{eq5} for this reason.} The diffusion terms in \eqref{eq1}-\eqref{eq5} are weighted by living population $n$, as the model hypothesizes that diffusion of individuals is not homogeneous, but preferential and directly proportional to the population.
\par This model was shown in \cite{VLABHPRYV2020} to \eli{exhibit} reasonably good agreement with measured data for the region of Lombardy, Italy, with later works \cite{BP2021, jha2020bayesian, grave2021assessing} showing good agreement in other regions. Further numerical and mathematical aspects were investigated in \cite{VLABHPRYV2020Due}, where it was also shown that models of this type can be put in the framework of continuum mechanics, and interpreted a balance of forces. Although the model \alex{demonstrates} acceptable agreement with reality and operates under sound physical assumptions, it relies extensively on unknown data. In particular, the exposed compartment (which also corresponds to the asymptomatic compartment) $e$ is difficult, if not impossible, to quantify with accuracy. We therefore propose the following modified model, which uses a \textit{delay differential equation} (DDE) formulation:

\begin{align}
\begin{split}\label{eq1Del} \partial_t s(t) &= \alpha n(t) - \left(1-\frac{A}{n(t)}\right)\beta_e s(t)i(t) - \left(1-\frac{A}{n(t)}\right)\beta_i s(t) i(t-\sigma) \\&\quad- \mu s(t)  + \nabla\cdot\left(n(t)\, \nu_s \nabla s(t)\right)\end{split} \\
\begin{split}\label{eq3Del} \partial_t i(t) &=  \left(1-\frac{A}{n(t)}\right)\beta_e s(t)i(t) + \left(1-\frac{A}{n(t)}\right)\beta_i s(t) i(t-\sigma) \\ &\quad- \phi_d i(t-\sigma) - \phi_r i(t-\sigma) - \mu i(t)  + \nabla \cdot\left(n(t) \nu_i \nabla i(t) \right)\end{split}\\
\label{eq4Del} \partial_t r(t) &= \phi_r i(t-\sigma) -   \mu r(t)  + \nabla \cdot\left(n(t)\, \nu_r \nabla r(t) \right) \\
\label{eq5Del} \partial_t d(t) &= \phi_d i(t-\sigma).
\end{align}
\par The corresponding ODE version of (\ref{eq1Del})-(\ref{eq5Del}) reads:
\begin{align}
\label{eq1DelODE} \dot{s}(t) &= \alpha n(t) - \beta_e s(t)i(t) - \beta_i s(t) i(t-\sigma)- \mu s(t)  \\
\label{eq2DelODE} \dot{i}(t) &= \beta_e s(t)i(t) + \beta_i s(t) i(t-\sigma) - \phi_d i(t-\sigma) - \phi_r i(t-\sigma) - \mu i(t) \\
\label{eq3DelODE} \dot{r}(t) &= \phi_r i(t-\sigma) -   \mu r(t)   \\
\label{eq4DelODE} \dot{d}(t) &= \phi_d i(t-\sigma).
\end{align}
We acknowledge a slight abuse of notation as, strictly speaking, $\sigma$ is the inverse of the corresponding value in system (\ref{eq1})-(\ref{eq5}), \eqref{eq1ODE}-\eqref{eq5ODE}. In general, the delay may be state-dependent; however, for the current work, we will assume that they are constant in order to simplify our analysis and computations. Note also that the definition of $n$ is now:
\begin{align}\label{nDDE}
n=s+i+r.	
\end{align}

\par As one may observe, the first major difference between the systems (\ref{eq1})-(\ref{eq5}) and (\ref{eq1Del})-(\ref{eq5Del}) is in the influence of the incubation period. Rather than include the exposed compartment $e$, the incubation period is incorporated into the system as a delay term on the infected compartment. A result of this choice is that in (\ref{eq1Del})-(\ref{eq5Del}), asymptomatic individuals are no longer specifically accounted for; all infected persons are considered equally. For the specific case of COVID-19, this may be a more reasonable assumption at this point in time, as testing protocols have improved and larger portions of asympomatic patients are now detected  \cite{kronbichler2020asymptomatic, Schneble2020}.
\par The second major difference between \eqref{eq1}-\eqref{eq5} and \eqref{eq1Del}-\eqref{eq5Del} is the evolution of the recovered compartment $r$ and deceased compartment $d$. As formulated in \eqref{eq1}-\eqref{eq5}, all members of the infected compartment $i$ are equally likely to die or recover at the same time; it does not make any distinction on these patients based on time of infection. In contrast, the recovery and mortality rates in \eqref{eq1Del}-\eqref{eq4Del} are delay-dependent, evolve according to the infected population at a previous point in time. This is a more realistic representation of epidemic dynamics, and may be useful when considering the allocation of public health resources. 
\subsection{Relationship between the PDE and ODE}
\alex{The \eli{presented delayed} PDE model \eqref{eq1Del}-\eqref{eq5Del} and ODE model  \eqref{eq1DelODE}-\eqref{eq4DelODE} are related, \eli{since they are used to describe the same phenomenon}, but differ due to the presence of diffusion and the Allee term $A$. This spatial information obviously gives the PDE model a richer descriptive capacity; however, it is also the case that \eli{the mathematical analysis and numerical} simulations using the PDE require significantly more effort, both in terms of computational time and the complexity of the model implementation. For this reason, the question of when the ODE model may provide a reasonable surrogate for the PDE is an important one.}
\par \nic{As our analysis in the following sections will demonstrate,
close to the zero equilibrium, that is when all quantities are small, the spatial (diffusive) terms, which are quadratic,
are negligible and the PDE becomes an ODE with delay terms. In the case when $A=0$, which is considered for example in our 1D simulations, a complete stability analysis of the equation allows to obtain sharp rigorous stability bounds which emphasize the dependence of stability of the steady state with respect to the delay. In the case when $A$ is nonzero, such stability bounds can be interpreted in an approximate way.  
For reasonably small variations in local population densities, the ODE solution may still well-approximate the spatially integrated PDE solution. Thus, if the spatial transients are not considered important, as may be the case in certain applications, the ODE model may be a more practical choice than the PDE model for its computational convenience. We will discuss the relationship between the two models formally in the analysis section.}

\eli{Moreover, we will qualitatively compare the behavior of the solutions obtained with the two models, in order to illustrate the theoretical results in the numerical experiments sections.} 
\section{Analysis}
In this section we will analyze the DDE models \eqref{eq1Del}-\eqref{eq5Del}, \eqref{eq1DelODE}-\eqref{eq4DelODE} mathematically. In particular, we examine the equilibrium solutions of \eqref{eq1DelODE}-\eqref{eq4DelODE} and \eqref{eq1Del}-\eqref{eq5Del} for the case $A=0$ and their stability properties. We then proceed to analyze the scalar linear equation associated to \eqref{eq1DelODE}-\eqref{eq4DelODE},  \eqref{eq1Del}-\eqref{eq5Del} with $A=0$, deriving stability conditions in terms of the physical parameters. \alex{We then examine the general case of \eqref{eq1Del}-\eqref{eq5Del} for $A \neq 0$, and analyze the impact of the Allee term on the stability behavior.}

\subsection*{Equilibria and their stability}

It is straightforward to note that the only equilibrium of (\ref{eq1DelODE})-(\ref{eq4DelODE}) is
\begin{equation*}
\left( s^*, i^*, d^*, r^* \right) = \left( 0, 0, 0, 0 \right).
\end{equation*}

The linearized system is given by
\begin{align}
\label{linearized1}\dot{s}(t) &= \alpha\,n(t) -  \mu\,s(t)
\\
\label{linearized2} \dot{i}(t) &= -\left( \phi_d + \phi_r \right) i(t-\sigma) - \mu\,i(t) \\
\label{linearized3} \dot{r}(t) &= \phi_r i(t-\sigma) - \mu\,r(t) \\
\label{linearized4}\dot{d}(t) &= \phi_d i(t-\sigma)
\end{align}
By adding equations \eqref{linearized2}-\eqref{linearized4} to \eqref{linearized1} and making use of the variable $n$ (as defined in \eqref{nDDE}), instead of $s$ we may replace \eqref{linearized1} with
\begin{eqnarray}
\dot{n}(t) & = & \left( \alpha -  \mu\right)\,n(t) + \mu\,d(t)
\label{eq:ddelin2}
\end{eqnarray}

Due to the last equation (\ref{linearized4}) of the DDE system, no contractivity arguments can be used to
infer about the asymptotic stability of the solution.

\subsection*{Stability of the scalar linear delay equation}

In order to state the stability theorem we need to consider the scalar equation
\begin{equation}
\dot{y}(t) = \alpha y(t) + \beta y(t-\tau)
\label{dde}
\end{equation}
with the delay $\tau$ an arbitrary but fixed positive constant.

By the change of variables $t=t/\tau$, $a = \alpha \tau$, $b = \beta \tau$ we are 
led to the equation
\begin{equation}
\dot{y}(t) = a y(t) + b y(t-1).
\label{dde0}
\end{equation}
The analysis of the characteristic equation,
\begin{equation}
\lambda = a + b\,{\rm e}^{-\lambda}
\label{eq:char}
\end{equation}
which for $b \neq 0$ possesses infinitely many solutions $\{ \lambda_k \}_{k=1}^{\infty}$, gives indications about the asymptotic behavior of the 
solution of \eqref{dde0}. 
The general solution is then obtained as a sum of exponentials,
\begin{equation*}
y(t) = \sum_k c_k {\rm e}^{\lambda_k t}.
\end{equation*} 
The zeros of \eqref{eq:char} are plotted in the left picture of Figure \ref{fig:stab} for the case $(a,b)=(0.5,-1)$. 
Notice that they all lie in the left half-plane, so that the solution will tend to zero for $t \to \infty$ (despite a positive $a$).

We provide the stability region in the right picture of Figure \ref{fig:stab} (see e.g \cite{BZ03}).
\begin{figure}
\centering
\begin{tabular}{c}
\hspace{-0.1in}
\includegraphics[height=2.05in]{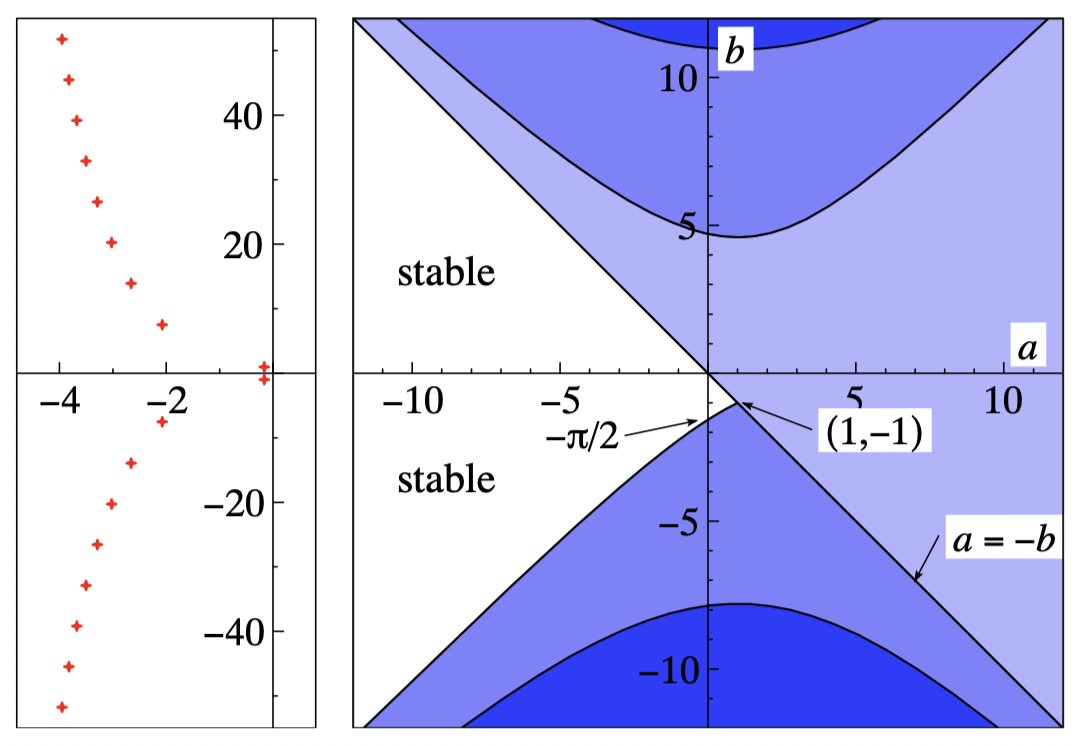}
\end{tabular}
\caption{Asymptotic stability region of equation \eqref{dde0} in the $(a,b)$-plane (right); roots  of the characteristic equation (left) for $a=0.5, b=-1$.
\label{fig:stab}}
\end{figure}
Note that if $a=0$ then the zero solution is asymptotically stable for $b \in (-\pi/2, 0)$ and equivalently
- referring to equation \eqref{dde} - it is asymptotically stable if $b \tau \in (-\pi/2, 0)$. 
The transcendental curve bounding the stability region is expressed in parametric form as
\begin{equation*}
\left( a(\phi), b(\phi) \right) = \left( \phi \cot(\phi), {}-\frac{\phi}{\sin(\phi)} \right), \qquad \phi \in (0,\pi) 
\end{equation*} 
which might be expressed as a monotonically increasing function $b(a)$. Asymptotically the curve approaches the line $b=-a$;
as $\phi \to 0$ it tends to the point $(1,-1)$ and at $\phi=\pi/2$ it crosses the point $(0,-\pi/2)$. 

\begin{theorem}\label{thm1}
Assume 
\begin{itemize}
\item[(i) ]  $\alpha - \mu < 0$ or $\alpha = \mu = 0$; 
\vspace{2mm}
\item[(ii) ] $\displaystyle \phi_d + \phi_r < \frac{\pi}{2 \sigma}$.
\end{itemize}
Then the zero equilibrium of \eqref{linearized1}-\eqref{linearized4}, \eqref{eq:ddelin2} is stable.
\end{theorem}

\begin{proof}
Consider first the DDE
\begin{equation*}
\dot{i}(t) =  - \mu\,i(t) -\left( \phi_d + \phi_r \right) i(t-\sigma)
\end{equation*}
which has the form \eqref{dde}. By non-negativity of $\mu$ we have that if 
\[
( \phi_d + \phi_r ) \sigma < \frac{\pi}{2} 
\]
the characteristic equation has all roots in the negative complex plane so that $\lim_{t \to \infty} i(t) = 0$. 
Moreover it can be shown that the decay is exponential (see \cite{DGVW95}).

Next look at
\begin{equation*}
\dot{d}(t) = \phi_d i(t-\sigma)
\end{equation*}
from which we obtain boundedness of $d$ as a consequence of the exponential decay of $i(\cdot)$.
Similarly, looking at
\begin{equation*}
\dot{r}(t) = - \mu\,r(t) + \phi_r i(t-\sigma) 
\end{equation*}
we have exponential decay if $\mu > 0$ and simply boundedness if $\mu=0$.
Finally, for what concerns the equation
\begin{equation*}
\dot{n}(t) = \left(\alpha -  \mu \right)\,n(t) + \mu\,d(t)
\end{equation*}
we conclude in a similar way, which means that if $\alpha-\mu < 0$ $n(t) \to 0$ as $t \to \infty$;
otherwise if $\alpha=\mu=0$ the stability statement holds trivially.
\end{proof}

\subsection{Stability of the equilibrium of the PDE}

The analysis of the PDE \alex{\eqref{eq1Del}-\eqref{eq5Del}} when $A=0$ leads to the same stability Theorem \ref{thm1}. This \alex{follows immediately from the fact that}, in the linearization
of \eqref{eq1Del},\eqref{eq3Del} and \eqref{eq4Del},
the terms 
$\nabla\cdot\left(n(t)\, \nu_s \nabla s(t)\right)$
and
$\nabla \cdot\left(n(t) \nu_i \nabla i(t) \right)$
are quadratic and hence do not \alex{influence the analysis}. As a consequence, the linearized system is formally analogous to the one obtained for the simpler DDE model \alex{\eqref{eq1DelODE}-\eqref{eq4DelODE}}.

More in detail, looking at the PDE-based model we have that, at the equilibrium
$\left( s^*, i^*, d^*, r^* \right) = \left( 0, 0, 0, 0 \right)$.
Neglecting second order terms gives the system:
\begin{eqnarray}
\label{linearizedPDE1}
\partial_t s(t) &=& \alpha n(t) - \left(1-\frac{A}{n(t)}\right) s(t) \left(\beta_e i(t) + \beta_i i(t-\sigma) \right)  - \mu s(t)   
\\
\label{linearizedPDE2}
\partial_t i(t) &=&  \left(1-\frac{A}{n(t)}\right) s(t) \left(\beta_e i(t) + \beta_i i(t-\sigma) \right) 
- \left( \phi_d i + \phi_r \right) i(t-\sigma) - \mu i(t)  
\\
\label{linearizedPDE3}
\partial_t r(t) &=& \phi_r i(t-\sigma) -   \mu r(t)   
\\
\label{linearizedPDE4}
\partial_t d(t) &=& \phi_d i(t-\sigma).
\end{eqnarray}
\noindent Observe that:

\begin{itemize}
\item[(a) ] If $A = 0$ we formally reobtain the same system \eqref{linearized1}-\eqref{linearized4} 
so that Theorem \ref{thm1} applies unchanged. \alex{In fact, in such a case, we may view the linearized system \eqref{linearized1}-\eqref{linearized4} as the system \eqref{linearizedPDE1}-\eqref{linearizedPDE4} integrated in space.}

\item[(b) ] The case $A \neq 0$ is more \alex{involved.}

By adding equations \eqref{linearizedPDE2}-\eqref{linearizedPDE4} to \eqref{linearizedPDE1} and making use of the variable $n$ instead of $s$, we may replace \eqref{linearizedPDE1} with:
\begin{eqnarray}
\nonumber
\partial_t n(t) & = & \left( \alpha -  \mu\right)\,n(t) + \mu\,d(t).
\label{eq:pdelin2}
\end{eqnarray}
Since $n(t) = s(t) + i(t) + r(t) $ and $s,i$ and $r$ are non-negative, we have that:
\begin{equation} \label{eq:delta}
\delta(t) := \frac{s(t)}{n(t)} \in [0,1] \qquad \forall t.
\end{equation}
which allows us to rewrite \eqref{linearizedPDE1}-\eqref{linearizedPDE4} as:
\begin{eqnarray}
\partial_t s(t) &=& \alpha n(t) + A \beta_e \delta (t) i(t) + A \beta_i \delta(t) i(t-\sigma) - \mu s(t)   
\nonumber
\\
\partial_t i(t) &=&  -A \delta(t) \left(\beta_e i(t) + \beta_i i(t-\sigma) \right) - \left( \phi_d + \phi_r \right) i(t-\sigma) - \mu i(t)
\nonumber
\\
&& \label{eq:pdelin3}
\\
\partial_t r(t) &=& \phi_r i(t-\sigma) -   \mu r(t)   
\nonumber
\\
\partial_t d(t) &=& \phi_d i(t-\sigma).
\nonumber
\end{eqnarray}

Looking at the second equation in \eqref{eq:pdelin3} we recognize a DDE of the form
\begin{equation}
\partial_t i(t) = a(t) i(t) + b(t) i(t-\sigma).
\label{eq:lddena}
\end{equation}
A well-known asymptotic stability condition for \eqref{eq:lddena}
is given by (see e.g. \cite{BZ03})
\begin{equation*}
a(t) + |b(t)| < 0 \qquad \forall t,
\end{equation*}
which yields:
\begin{equation*}
-\mu -A \left(\beta_e - \beta_i\right) \delta(t) + \phi_d + \phi_r < 0. 
\end{equation*}
In the usual \alex{case where} $\beta_e \ge \beta_i$, we obtain that the condition 
\begin{equation*}
-\mu + \phi_d + \phi_r < 0
\end{equation*}
implies asymptotic stability (of contractive type) of the solution $i$ independently of the delay $\sigma$.

We remark that this is a sufficient and not necessary condition for a stronger type of asymptotic stability, namely contractivity, 
of the solution.

Under this condition we have that $\lim\limits_{t \rightarrow \infty} i(t) = 0$.
The analysis of the remaining equations for the variables $r,d$ and $n$ is 
analogous to the one provided in the proof of Theorem \ref{thm1}.

\end{itemize}

Note that if we could treat $\delta(t)$ as a constant (see \eqref{eq:delta}),
we would get the same sufficient conditions to asymptotic stability provided
by Theorem \ref{thm1}, i.e.
\begin{itemize}
\item[(i) ]  $\alpha - \mu < 0$ or $\alpha = \mu = 0$; 
\hskip 1cm or \hskip 1cm
(ii) \ \ \ $\displaystyle \phi_d + \phi_r < \frac{\pi}{2 \sigma}$.
\end{itemize}
thus depending on the delay $\sigma$. 
In a regime where $\delta(t) \in [0,1]$ does not exhibit big variations, we expect that
such conditions continue to hold true, at least approximately.

\subsection{Comments}

In some cases we observe non-physical slightly negative values of the modeled quantities. This is due to the fact
that when 
\[
\nu=0 \quad \mbox{and} \quad 
\displaystyle \phi_d + \phi_r \ge \frac{1}{e \sigma},
\]
the rightmost roots of the characteristic equation are complex conjugate. This is easily seen observing that the equation $\lambda = b\,\mathrm{e}^{-\lambda}$ has no real roots if $b < -\frac{1}{\mathrm{e}}$.

Instead, when
$\displaystyle \phi_d + \phi_r < \frac{1}{e \sigma}$ a real root dominates.
However, since the oscillations occur when the solution approaches the steady state in the asymptotically stable regime, we may overlook the potential misbehavior.

\begin{remark} \rm
If $\mu < 0$ then the stability bound $\phi_d + \phi_r < \frac{M}{\sigma}$ can be made larger; that is,
$M$ increases as $\mu$ increases (see Figure \ref{fig:stab}).
\end{remark}

\begin{remark} \rm
We may understand the bound (ii) of Theorem 3.1 physically as a relationship between the removal rate, as governed by the parameters $\phi_d$ and $\phi_r$, and the time delay $\sigma$. This bound states that, for the equations to be stable, the rate of recovery and/or mortality from the disease must occur over a time scale sufficiently longer than the time delay $\sigma$ (recall that $\phi_d$ and $\phi_r$ have units 1/Time, while $\sigma$ has units Time). 
\end{remark}

\section{Numerical implementation and experiments}
In this section we will perform several numerical tests to evaluate various characteristics of the model and its numerical solution. In particular, we will perform the following experiments:
\begin{enumerate}
\item Several examples using the ODE version of the model 
\eqref{eq1DelODE}-\eqref{eq4DelODE}. Here, we observe the impact of the delay on aspects of the physical solution, including the effects on contagion and lockdown measures. We also examine the derived stability bounds and influence of different parameters.
\item A one-dimensional example using the PDE model \eqref{eq1Del}-\eqref{eq5Del} based on the simulation performed in \cite{VLABHPRYV2020Due,grave2020adaptive} for different values of $\sigma$ and problem parameters. These examples seek to examine the solution characteristics in both quantitative and qualitative aspects on an artificial problem which shares many characteristics with a real-world problem, but remains tractable and sufficiently simple to analyze in detail. We also seek to confirm the correspondence between simulations using the ODE and spatially integrated solutions of the PDE.
\item A two-dimensional example using the COVID-19 outbreak in Lombardy, Italy employing the PDE model \eqref{eq1Del}-\eqref{eq5Del}. This simulation is similar to the ones carried out in \cite{VLABHPRYV2020Due, VLABHPRYV2020}, which were well-validated against the measured data at the time of publication. This example is designed to show the viability of the delay-equation formulation in reproducing real-world data, as well as its performance when compared to non-delay models.
\end{enumerate}
\subsection{ODE Model}
In order to perform the simulations for the ODE model \eqref{eq1DelODE}-\eqref{eq4DelODE}, we employ the Matlab solver DDE23. As initial conditions, we set the total population $n=1000$, and as a historic function we choose  $i(t)=1$ for $t\in [-\sigma,0]$. Assuming $r(0)=0$ and $d(0)=0$, we end up with $s(0)=n-i$. The final time of the simulation is $t=267$ days. The parameter values are reported in Table \ref{tab:1DParametertTable}\footnote{For the ODE model, these values have been normalized by $n(0)$=1000, and accordingly has units of Days$^{-1}$.}. To observe the impact of the delay, we run the simulation for different values of $\sigma$: $\sigma=5, 10, 15$ and $20$ days.

We note that for increasing values of the delay the number of the infections is higher, i.e. in Fig. \ref{fig:ODE_Infected}(left) the infection peak for $\sigma=5$ (black line) is much lower than the peak for $\sigma=20$ (magenta line). Furthermore, we observe that the amplitude of the peak is larger for high values of the delay. On the other hand, if the delay is too high with respect to the parameters, we could obtain non physical solutions, as in Fig. \ref{fig:ODE_Infected}(left), where the infections becomes negative for $\sigma=20$. In the following we will investigate the impact of the government restrictions, i.e. the introduction of lockdowns, and the stability of the model from a numerical point of view.

\begin{table}
\begin{center}
\begin{tabular}{ |c|c|c|c| } 
\hline
Parameter   &  Units  & Value \\
\hline\hline
$\parbetae$\footnotemark[1] & Persons$^{-1} \cdot $ Days$^{-1}$ & 9/40 \\ \hline
$\parbetai$\footnotemark[1] & Persons$^{-1} \cdot $ Days$^{-1}$ & 3/32\\ \hline
$\parphir$ &  Days$^{-1}$ & 1/32 \\ \hline
$\parphie$  &  Days$^{-1}$ & 1/8 \\ 
\hline
$\parphid$ &  Days$^{-1}$ & 3/640 \\ \hline
$\mu$ & Days$^{-1}$& 0\\ \hline
$\alpha$ &  Days$^{-1}$ & 0\\ \hline

 $\overline\nu_s^*$& Persons$^{-1} \cdot $ Days $^{-1}$  & 3.75$\cdot 10^{-5}$ \\ \hline
 $\overline\nu_e^*$& Persons$^{-1} \cdot $ Days $^{-1}$  & .75$\cdot 10^{-3}$ \\ \hline
 $\overline\nu_i^*$& Persons$^{-1} \cdot $ Days $^{-1}$  & .75$\cdot 10^{-10}$ \\ \hline
 $\overline\nu_r^*$& Persons$^{-1} \cdot $ Days $^{-1}$  & 3.75$\cdot 10^{-5}$ \\
 \hline

\end{tabular}
\caption{Parameter values for the ODE and 1D simulations. Note all values have been normalized in space by a characteristic length scale $L$, with this normalization reflected in the units.}
  \label{tab:1DParametertTable} 
\end{center}
\end{table}

\subsubsection{Effect of lockdowns}
Due to the relevance of the pandemic on the dailylife routine, we seek to observe the effect of government restrictions, i.e., lockdowns, on the evolution of the compartments. To this end, we run two cases, one without lockdowns, in which the contact rates $\beta$ are kept constant throughout the simulation, and one with lockdowns, in which we set $\beta=\beta/4$ at $t=30$ days. In fact, the aim of the lockdown is to reduce the contact rate.

In Fig. \ref{fig:ODE_Infected}, we show the evolution of the infected compartment in time, for the different values of $\sigma$ without (left) and with (right) restrictions. It is clear that the number of infections increases with the delay even in the lockdown situation, as expected. However the lockdown restriction reduces the number of infected people by about $20\%$. Moreover, we observe the same effect focusing on the deceased compartment. Indeed, looking at the deaths peak in Fig. \ref{fig:ODE_stab}, the deaths peak is lower in the lockdown situation.
\par We also observe that, with larger values for the delay, the effect of the lockdown measures is less readily observed. One may particularly see this in Fig. \ref{fig:ODE_Infected} on the right, where the decrease of infections as a result of the lockdown begins very quickly for $\sigma=5$, and progressively more slowly for larger values of $\sigma$. This is consistent with our expectations.
\begin{figure}\centering
  \includegraphics[width=.49\textwidth]{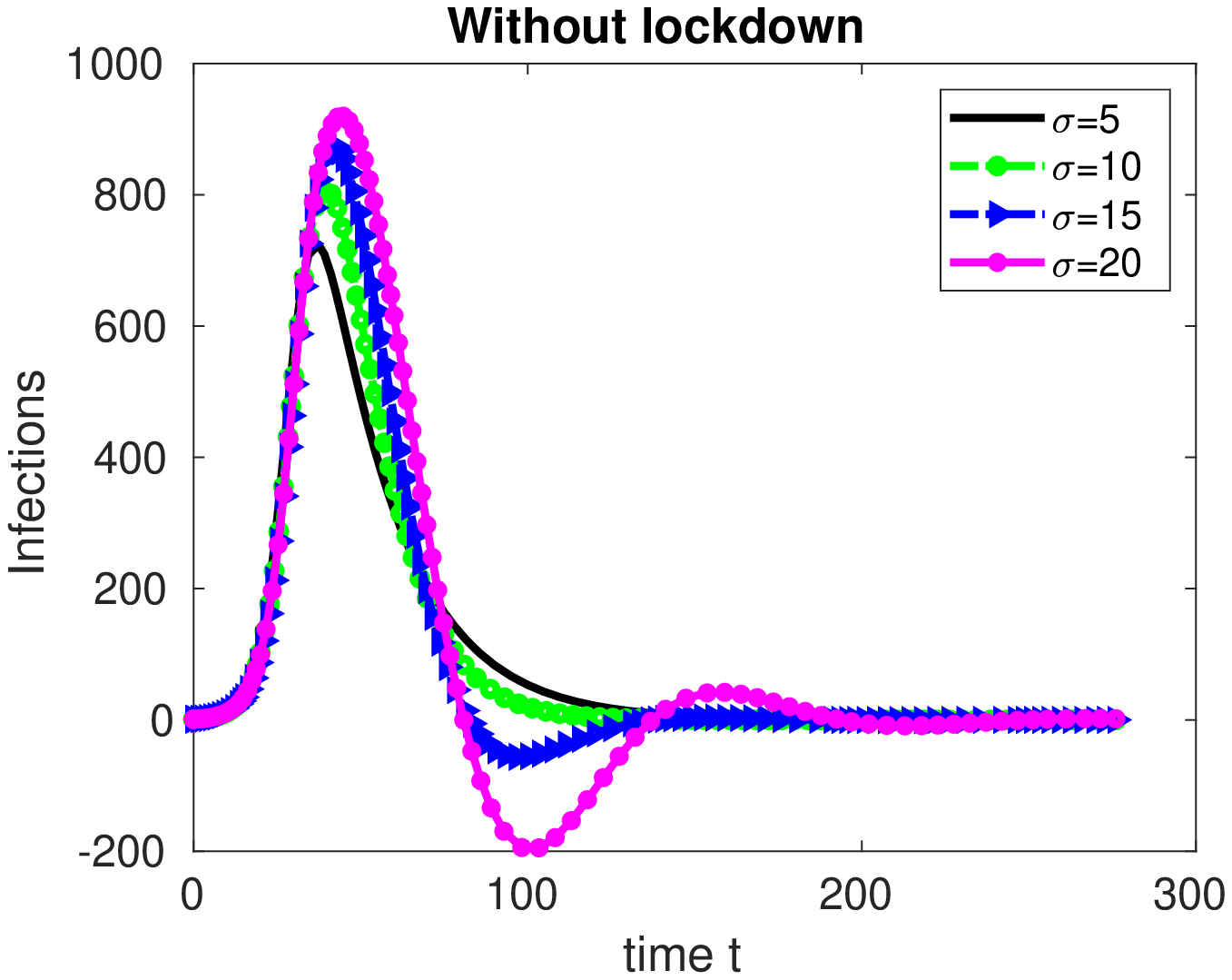} 
  \includegraphics[width=.49\textwidth]{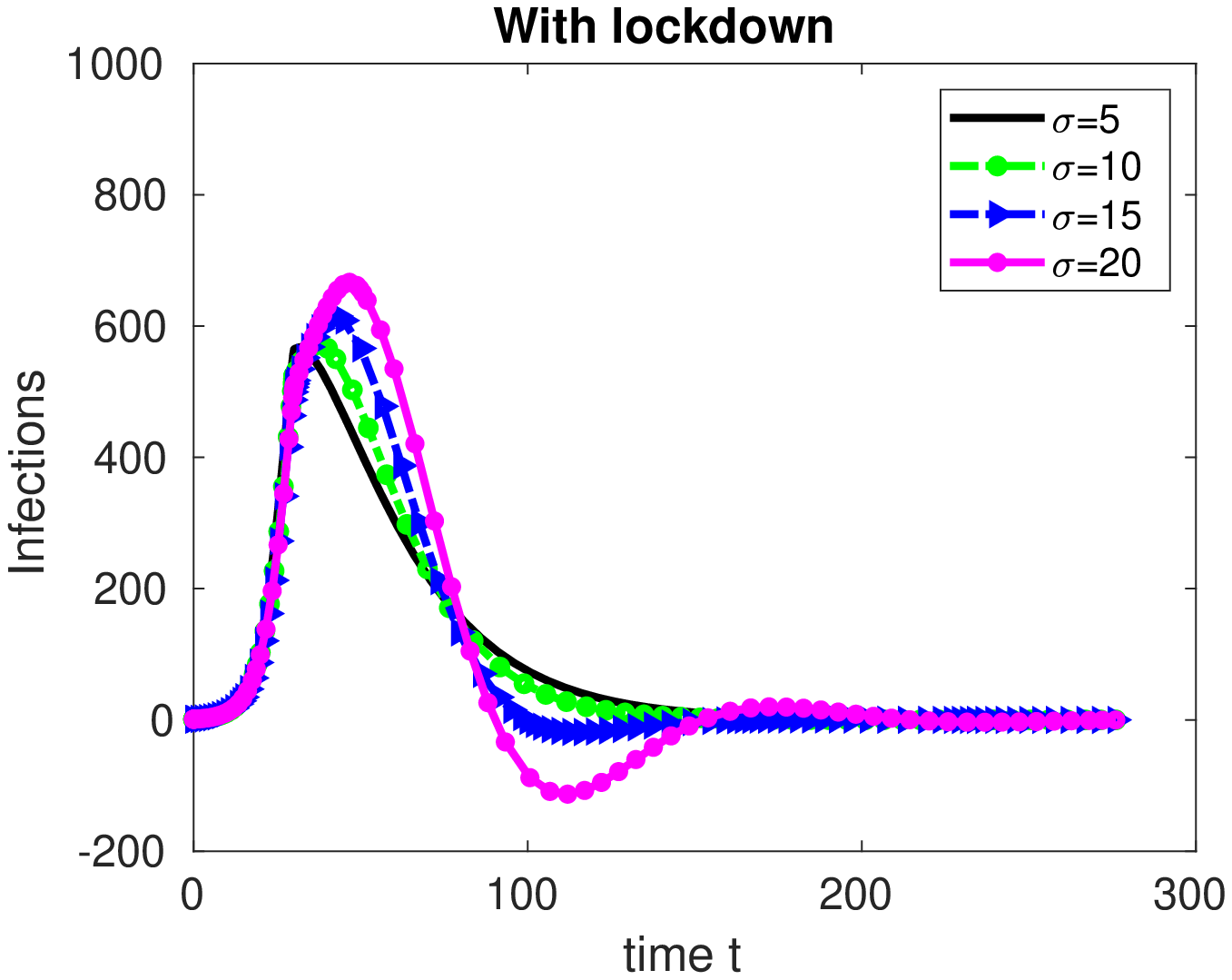} \caption{Total infected for $\phi_r$=1/32, $\phi_d$=3/640 for different delay values in the non lockdown (left) lockdown (right) case.}\label{fig:ODE_Infected}
\end{figure}
\begin{figure}\centering
  \includegraphics[width=.49\textwidth]{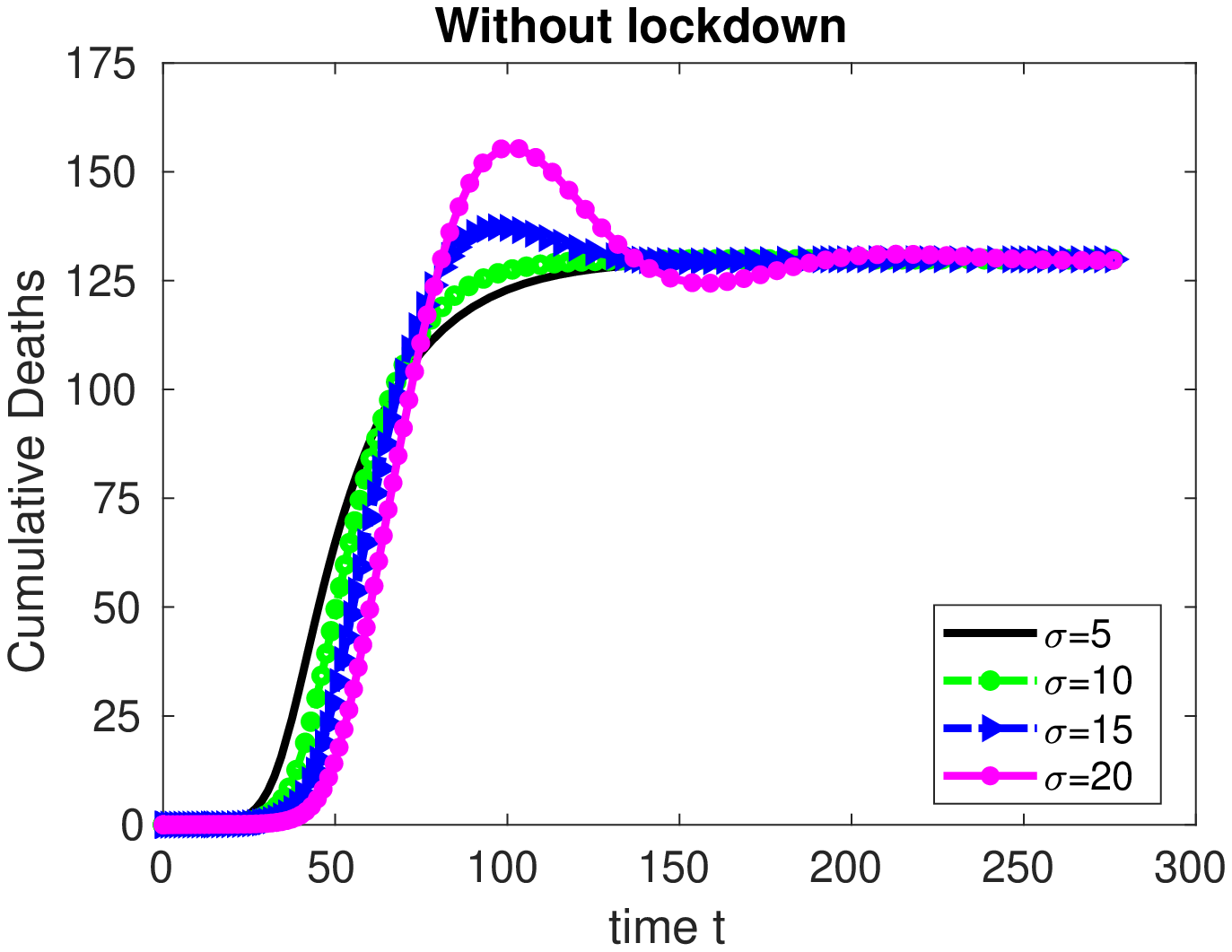}
  \includegraphics[width=.5\textwidth]{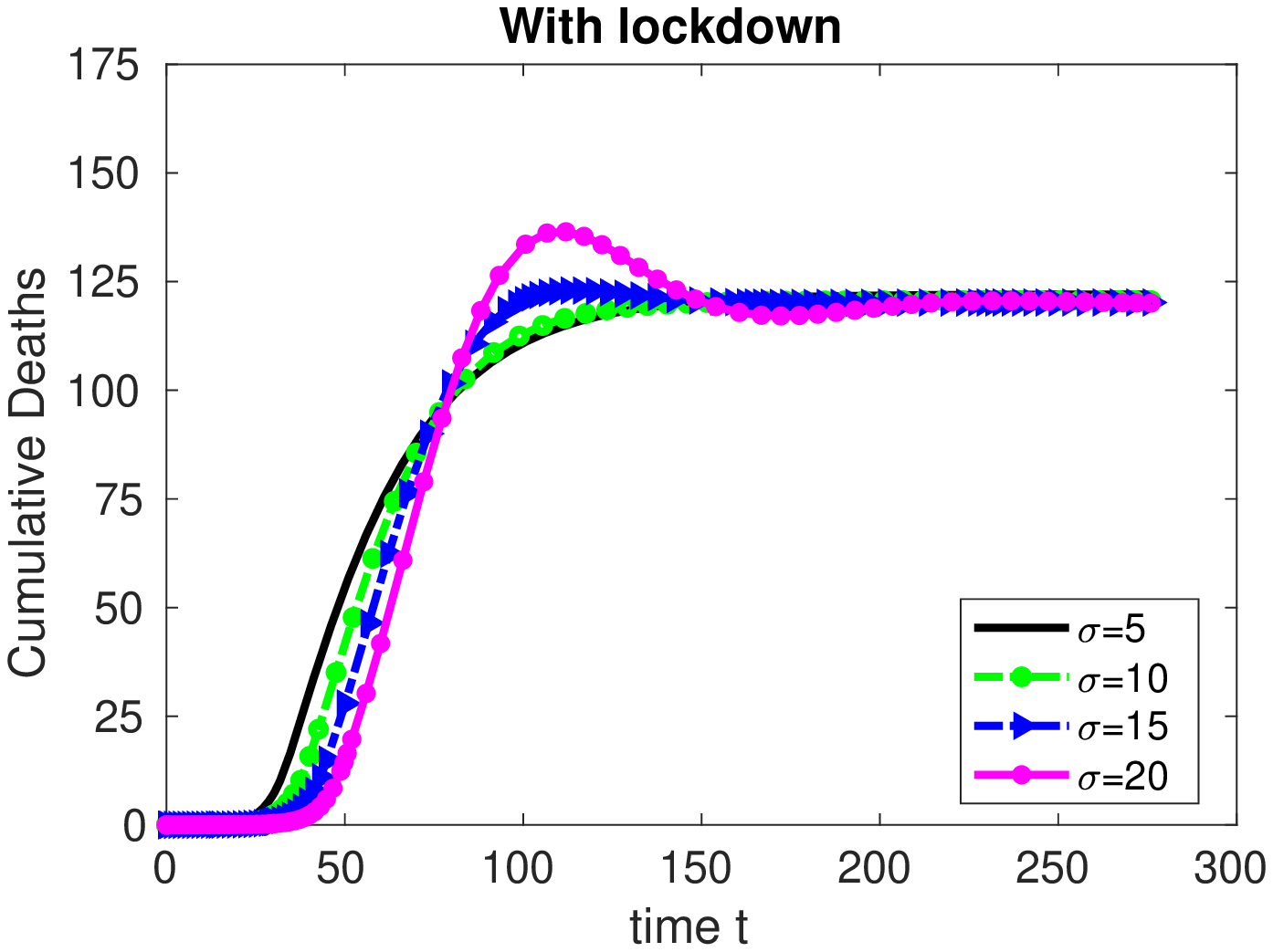}
  \caption{Total deceased for $\phi_r$=1/32, $\phi_d$=3/640 and for different delay values in the non lockdown (left) lockdown (right) case.}\label{fig:ODE_stab}
\end{figure}

\subsubsection{Stability of the ODE model}
We now investigate the numerical stability of the ODE model, and in particular, we seek to examine the validity of the bound in Theorem \ref{thm1} and verify it numerically. 
Looking at the deceased compartment, Fig. \ref{fig:ODE_stab} shows that the solution is stable for $\sigma=5,10,15,20$, since we choose the parameters according to the bounds $(i)-(ii)$ of the Theorem \ref{thm1}. Indeed, $\alpha=\mu=0$, $\phi_r=1/32, \phi_d=3/640$, which means that
	$$\frac{1}{32}+\frac{3}{640} \approx .0359 < \frac{\pi}{2\sigma}$$
for all our choices of the delay. 

On the other hand, modifying the parameters as $\phi_r=3/32$ and  $\phi_d=1/80$ we get an unstable solution for $\sigma=15$, as shown in Fig. \ref{fig:ODE_non_stab}. In fact for $\sigma=15$ the oscillations are increasing in time, while for $\sigma=10$ they are smearing out. This is also consistent with the analysis, as we may expect oscillations for larger values of $\sigma$, however, if the numerical bound is respected, these oscillations should stabilize and not affect the solution asymptotically.

\begin{figure}\centering
  \includegraphics[width=.7\textwidth]{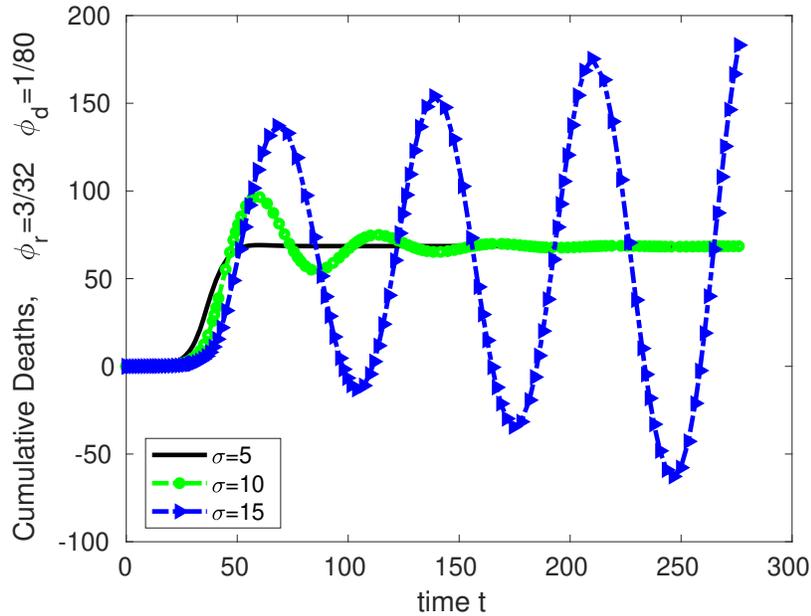} \caption{Total deceased for $\phi_r$=3/32, $\phi_d$=1/80, for different delay values.}\label{fig:ODE_non_stab}
\end{figure}

An interesting behavior that we also observe, is that if we choose $\phi_d=1/80$ and $\phi_r=28/320$ we obtain a periodic, non-physical behaviour of the solution for $\sigma=15$, Fig. \ref{fig:ODE_periodic}. As shown in the figure, we observe oscillations that neither increase nor decrease, instead demonstrating what appears to be a true periodic regime. Indeed the with these parameters we have:
$$\frac{28}{320}+\frac{1}{80} \approx \frac{\pi}{2\sigma}.$$ This suggests that, near the limit of the stability bound, solutions exhibit a periodic behavior. Considering that the oscillations decrease for $\phi_r$, $\phi_d$ sufficiently below the stability bound, and increase for $\phi_r$, $\phi_d$ sufficiently large, this behavior is perhaps to be expected. Whether this is a mere mathematical curiosity or perhaps indicates relevant biological information is not clear, and is potentially a subject of future investigation.
\begin{figure}\centering
\includegraphics[width=.49\textwidth]{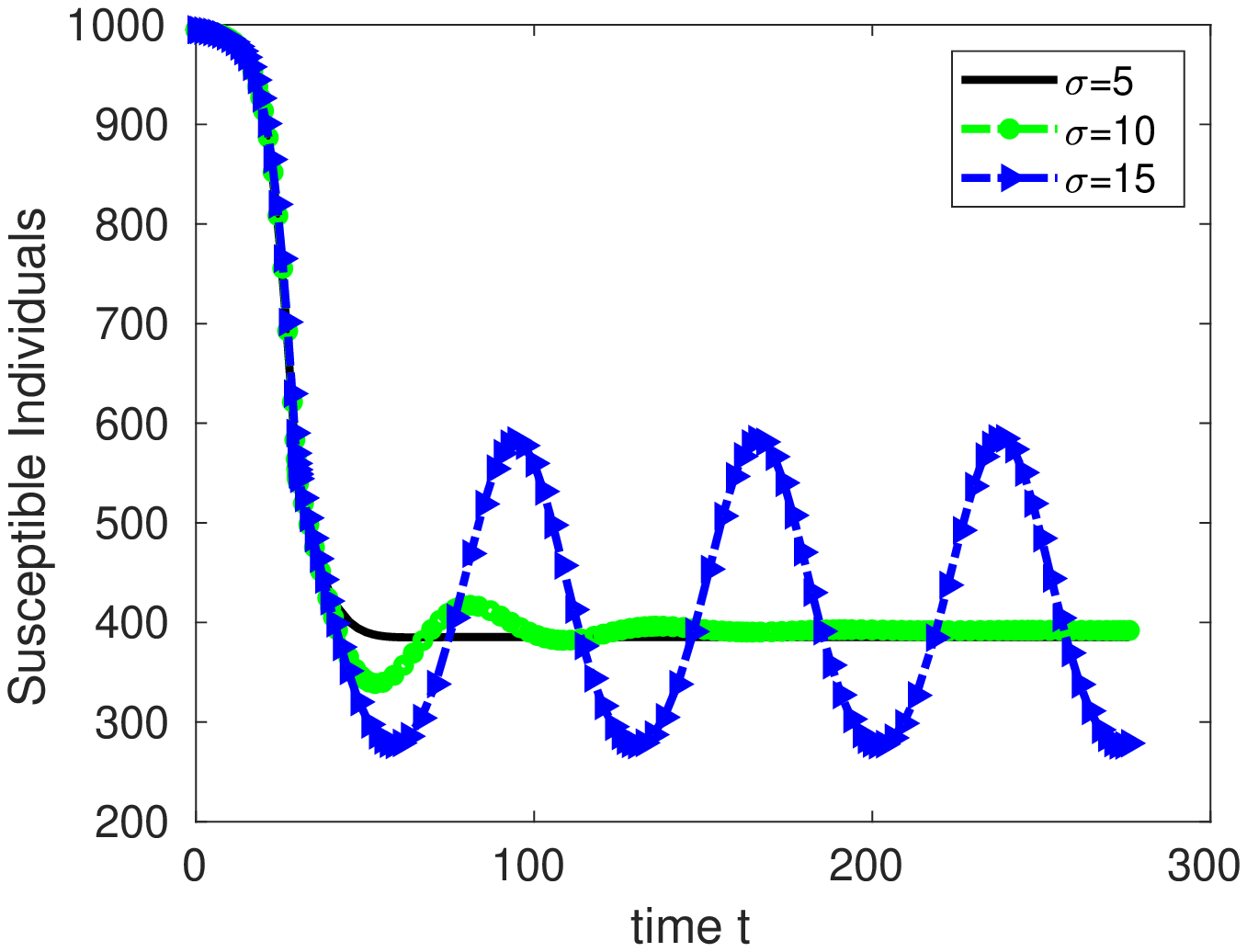}
\includegraphics[width=.49\textwidth]{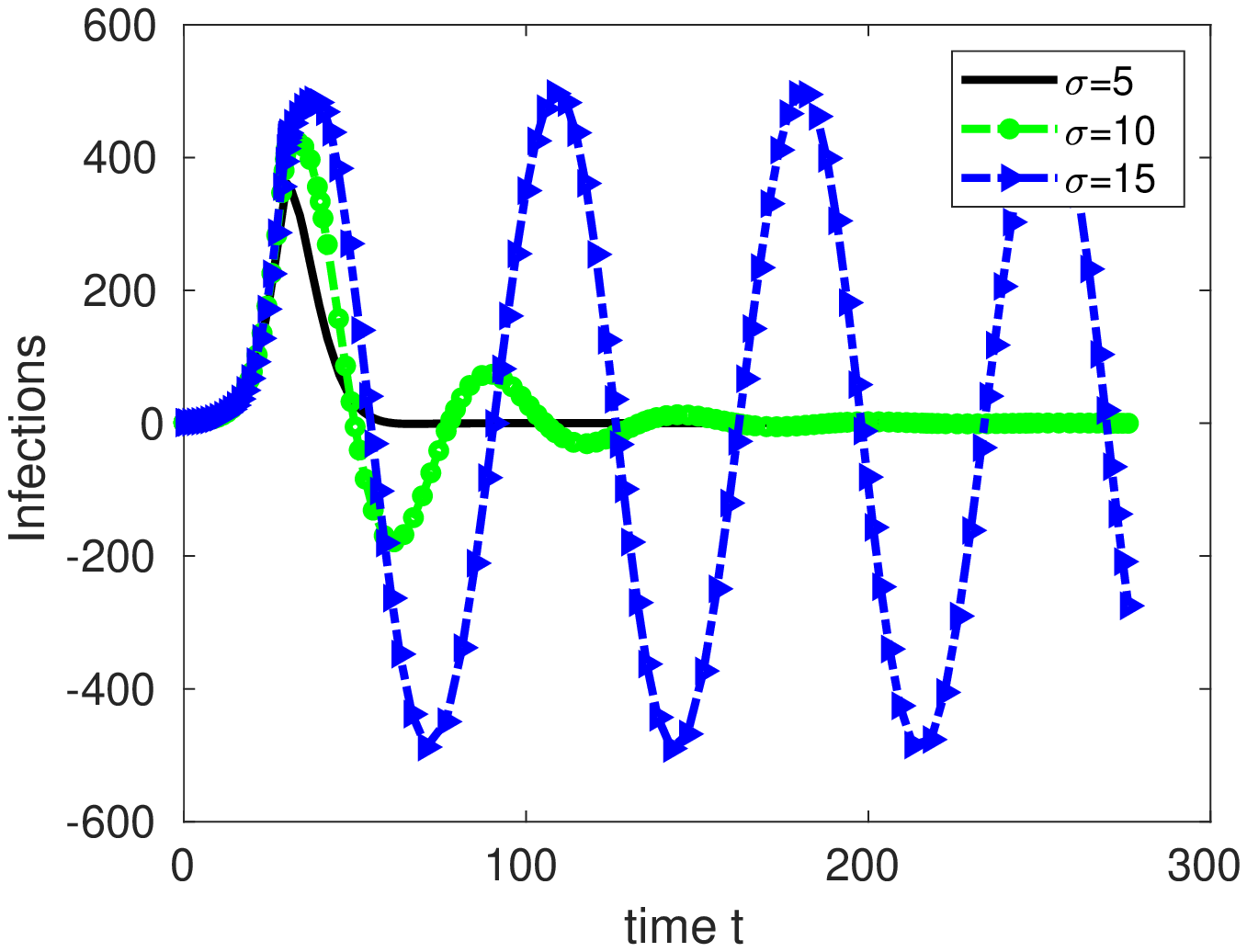}\\
\includegraphics[width=.49\textwidth]{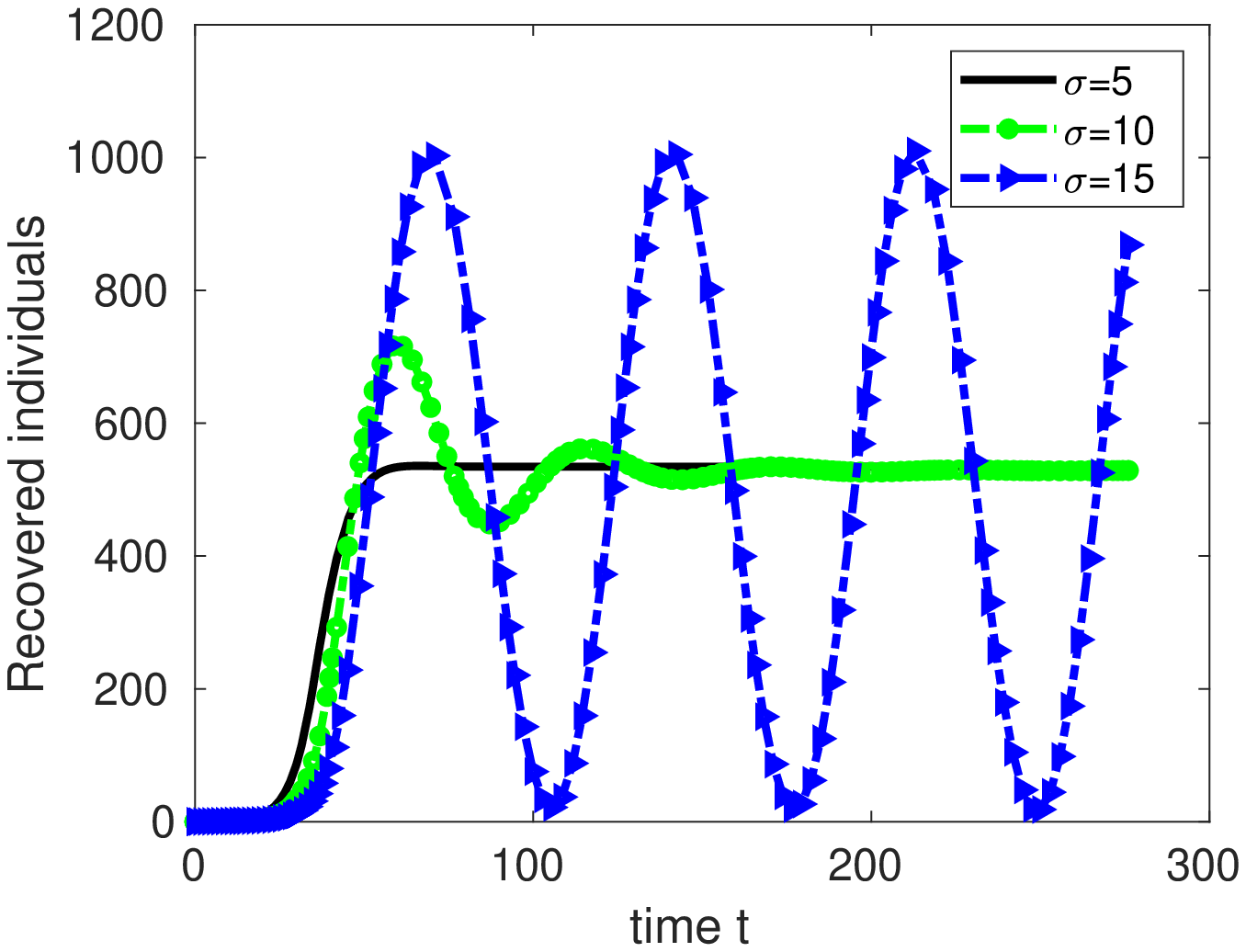}
\includegraphics[width=.49\textwidth]{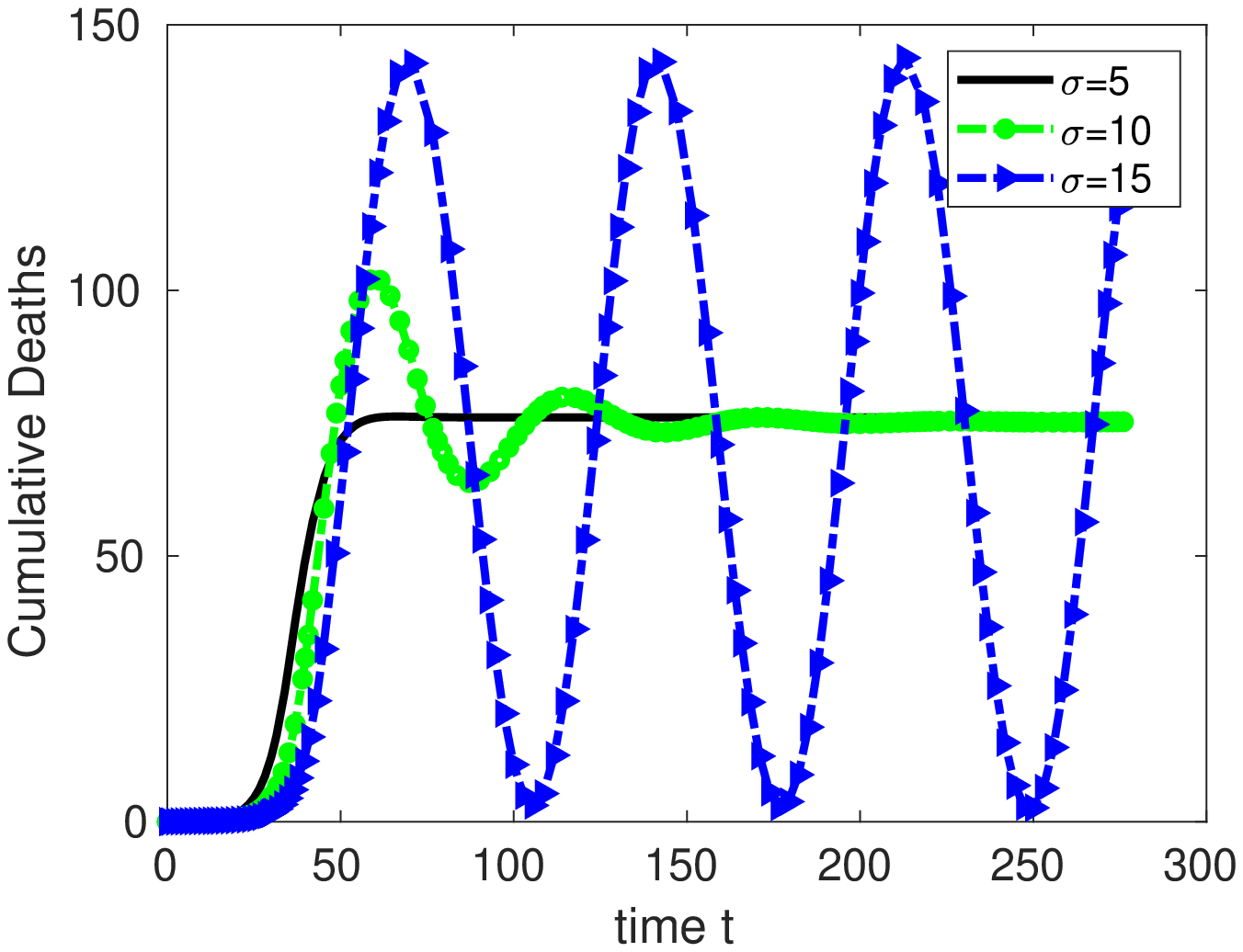} \caption{Evolution of susceptible compartment(top-left), infected compartment (top-right), recovered compartment (bottom-left) and cumulative death (bottom-right) for different delay values for $\phi_d=1/80$ and $\phi_r=28/320$.}\label{fig:ODE_periodic}
\end{figure}

\subsection{1D PDE Model}
In this example, we follow basic setup inspired by the one-dimensional example introduced in \cite{VLABHPRYV2020Due} and also performed in \cite{grave2020adaptive}. We will examine the behavior of the solution under various conditions, as well as the validity of the stability bound for the partial differential equation.
\subsubsection{Problem Setup}
\begin{figure}\centering
  \includegraphics[width=.7\textwidth]{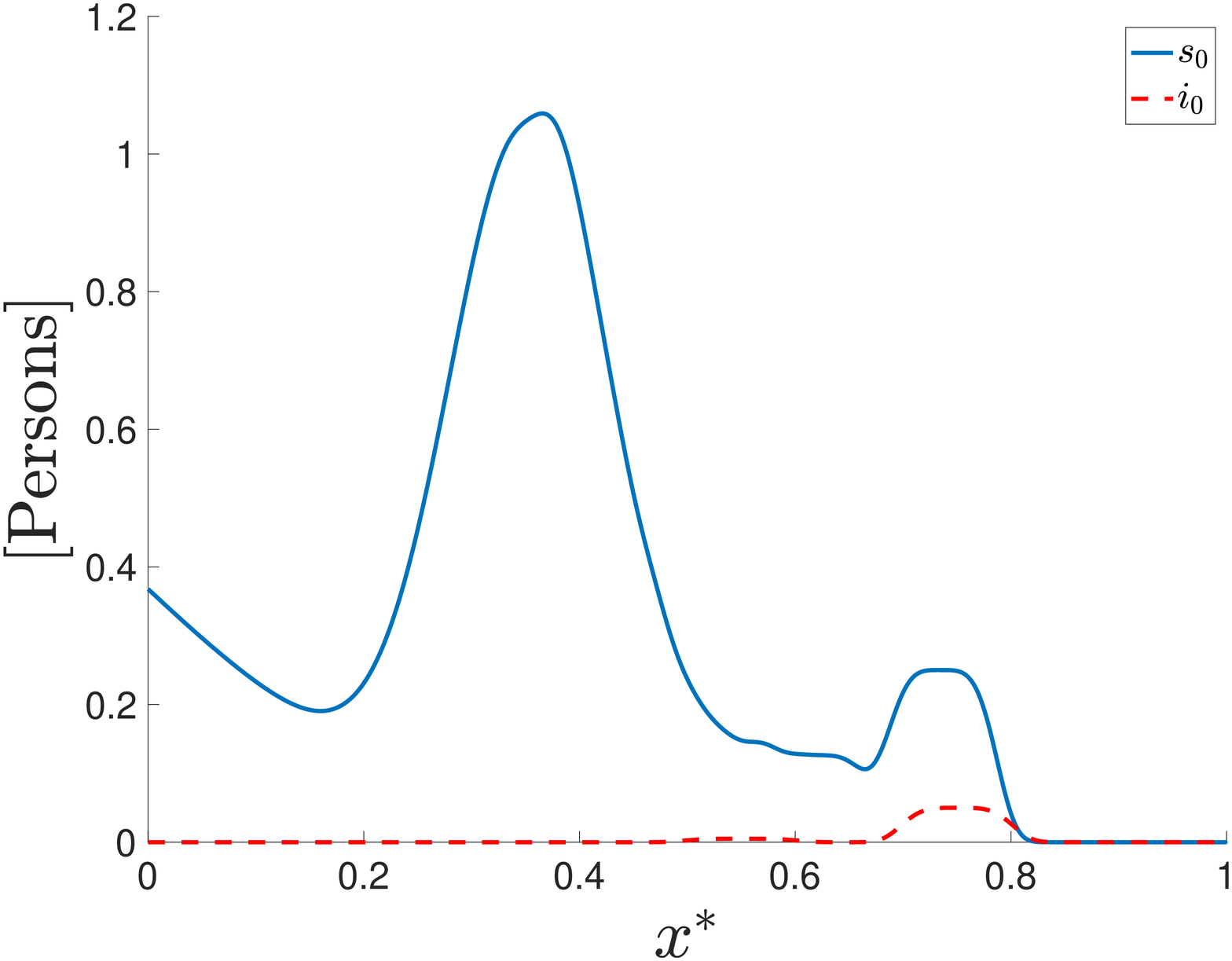} \caption{Initial conditions for the 1D PDE model. Qualitatively, the setup represents a population distribution with one major population center and one lesser population center, with an initial outbreak centered in the lesser population center.}\label{fig:InitialValues}
\end{figure}

For the initial conditions, we set $s(x^*,0) = s_0(x^*)$ and $e(x^*,0)=e_0(x^*)$ as follows
\begin{align}\label{susceptInit}
 s_0(x^*)&= e^{-(x^*+1)^4} + e^{-\frac{(x^*-.35)^2}{1e-2}} + 
 \frac{1}{8}\left(e^{-\frac{(x^*-.62)^4}{1e-5}} +
  e^{-\frac{(x^*-.52)^4}{1e-5}} + 
 e^{-\frac{(x^*-.42)^4}{1e-5}}\right) + 
 \frac{1}{4} e^{-\frac{(x^*-.735)^4}{1e-5}}, \\
	i_0(x^*) &= \frac{1}{20}e^{-\frac{(x^*-.75)^4}{1e-5}} +  \frac{1}{200}e^{-\frac{(x^*-.55)^4}{1e-5}}. \label{exposedInit}
\end{align}
\noindent These conditions are plotted in Fig. \ref{fig:InitialValues}. Qualitatively, they correspond to a population distribution with one major population center, one moderate population center, and one lesser population center, with an initial outbreak centered in the lesser population center.
\par For the parameters, we use the values shown in Table \ref{tab:1DParametertTable}. We discretize in space over the unit interval with $\Delta x= 1/2000$, and advance in time using the BDF2 scheme with $\Delta t =.25$. It was demonstrated in \cite{VLABHPRYV2020Due} that this spatiotemporal discretization resolves all dynamics satisfactorily. We note also that our choice of time-step ensures that we may use previously computed solutions for our delay terms, and there is no need for interpolation \cite{BZ03}.
\par We run the simulation for $t=267$ days for $\sigma=5,\,10,\,15,\,$ and $20$. We seek to examine the effect of $\sigma$ on contagion, but also on the efficacy of public health interventions. To this end, for each value of $\sigma$ we run the full time interval both with and without lockdown measures. For the case with no lockdown, we use the parameters shown in Table \ref{tab:1DParametertTable} over the entire time interval. For the case with lockdowns, at $t=140$ we multiply all diffusive terms $\nu$ by $1/2$, simulating restricted mobility, and contact terms $\beta$ by 1/4, corresponding to measures such as bar closures, mask wearing etc. We note that the problem setup is designed to resemble the ODE simulations in the preceding subsection. 
\subsubsection{Effect of lockdowns}
In Figs. \ref{fig:PlotSusceptible}, \ref{fig:PlotInfected}, \ref{fig:PlotRecovered}, \ref{fig:PlotDeceased} we show the total (i.e., integrated in space) susceptible, infected, recovered, and deceased compartments in time, respectively for each value of $\sigma$ and lockdown configuration. We see that longer $\sigma$ is associated with higher infective peaks, however cumulative deaths, long-term, are similar for all $\sigma$. This is consistent with our expectations, as reducing infective peaks does not necessarily correspond to fewer cases overall. 
\par More interesting is the observed effect of $\sigma$ on lockdown efficacy. Referring to Fig. \ref{fig:PlotDeceased}, the results show that longer $\sigma$ leads to a noticeable lag in the delay compartment; indeed, one begins to see a decrease in mortality as a result of the lockdown at a later date for larger $\sigma$. In the plot of the infected compartment in Fig. \ref{fig:PlotInfected}, the impact of lockdowns appears immediate. This is indeed what we expect, as the definition of this compartment includes pre-symptomatic patients; thus, while the effects are immediate by this definition, other indicators, such as the aforementioned deceased compartment, will lag in proportion to $\sigma$.
\par The delay $\sigma$ appears to not only affect the time at which the effect of lockdowns appear, but also the sharpness of such effects. While for larger $\sigma$, the deceased compartment begins to decrease later as compared to smaller values, once this decrease begins, it occurs much more rapidly. This is particularly visible in the infected compartment shown in Fig. \ref{fig:PlotInfected}; around $t=175$, the total number of infections is indeed \textit{larger} for smaller values of $\sigma$ .
\par An important dynamic that we notice for large $\sigma$ is the emergence of non-physical behavior, similar to those observed in the ODE case. This is particularly apparent for the case of $\sigma=20$, where we observe the infected compartment becoming negative, and decreases in the recovered and deceased compartments, which should be monotonic. While this behavior is non-physical, it is not mathematically inconsistent with the model behavior and does not represent instability as such. Indeed, to guarantee positivity, more stringent conditions on the relationship between $\phi_r$, $\phi_d$ and $\sigma$ are likely required. 

\begin{figure}\centering
  \includegraphics[width=.7\textwidth]{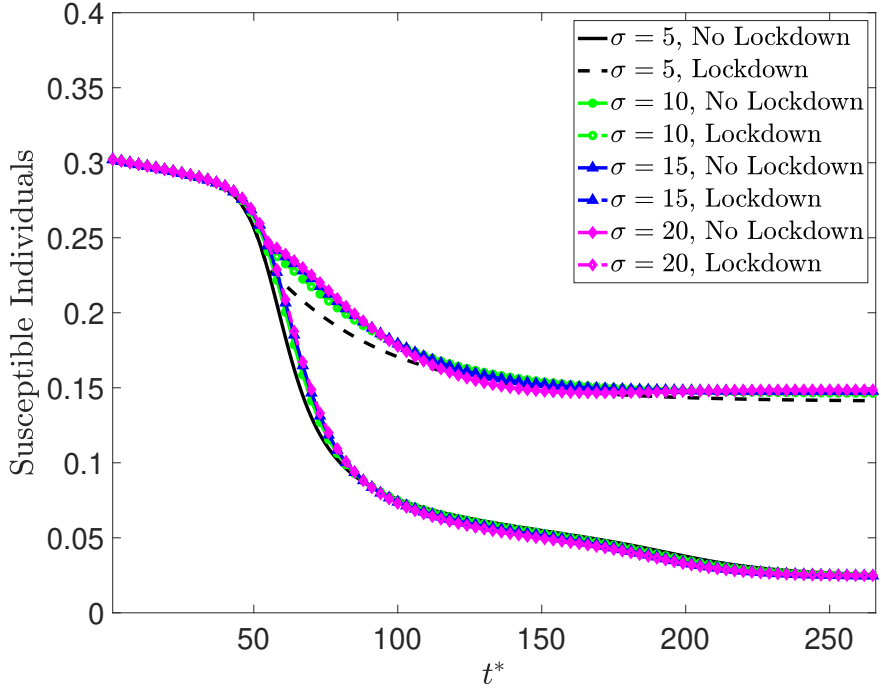} \caption{Total susceptible for $\phi_r$=1/32, $\phi_d$=3/640. We have stable and monotonic behavior in this compartment across all different cases. A lag is observed, as expected, for the different values of $\sigma$.}\label{fig:PlotSusceptible}
\end{figure}

\begin{figure}\centering
  \includegraphics[width=.7\textwidth]{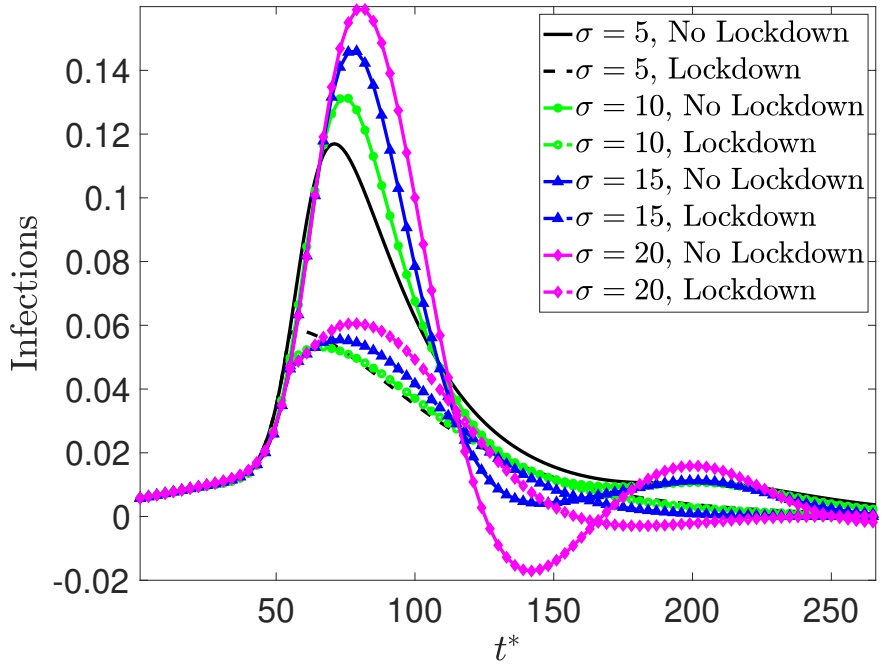} \caption{Total infected for $\phi_r$=1/32, $\phi_d$=3/640. A longer $\sigma$ is associated with more infections, but also a more dramatic decrease in infection after lockdowns are initiated. The impact of lockdowns in the infected compartment is visible more immediately, with the lag effect being more pronounced in the deceased compartment. We also observe clearly the non-physical behavior for $\sigma=20$ in both cases, with the infected compartment becoming negative.}\label{fig:PlotInfected}
\end{figure}

\begin{figure}\centering
  \includegraphics[width=.7\textwidth]{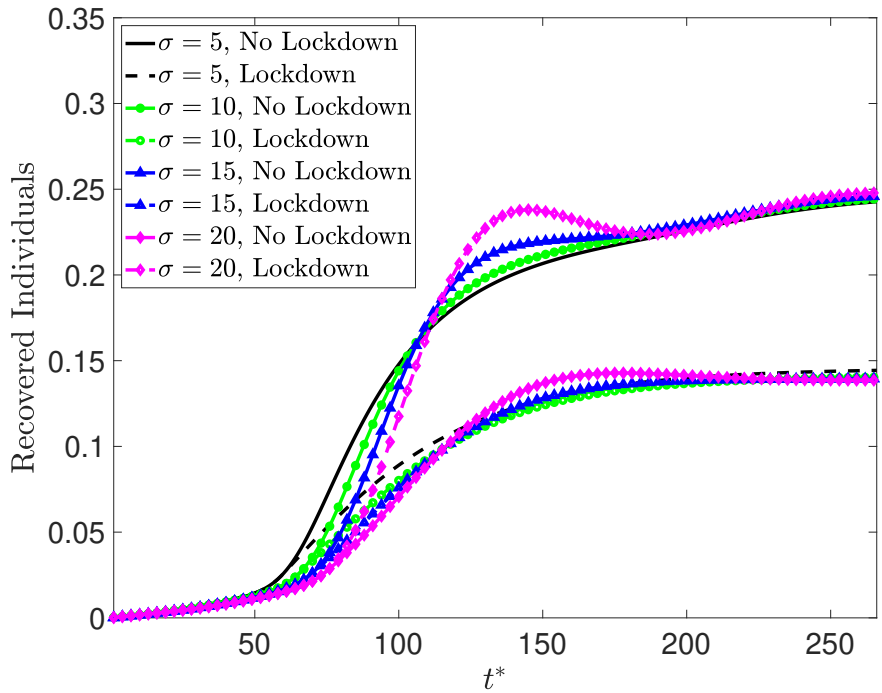} \caption{Total recovered for $\phi_r$=1/32, $\phi_d$=3/640.  Each case is stable, but we observe nonphysical behavior as the deceased compartment for $\sigma=20$, with both cases demonstrating noticeable non-monotonicity. The smaller amount of recovered individuals in the lockdown cases is explained by the reduced contagion overall.}\label{fig:PlotRecovered}
\end{figure}

\begin{figure}\centering
  \includegraphics[width=.7\textwidth]{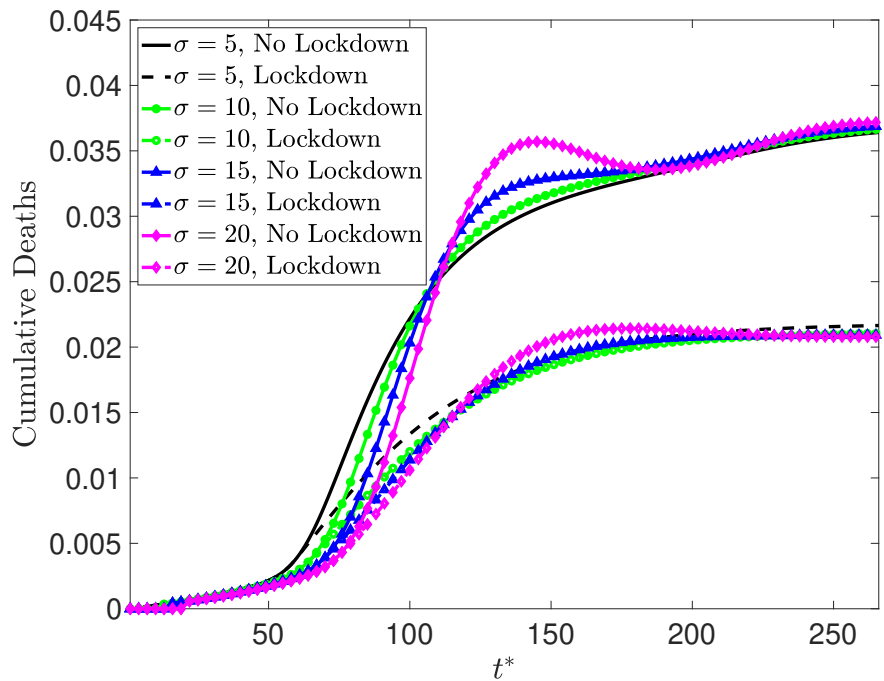} \caption{Total deceased for $\phi_r$=1/32, $\phi_d$=3/640. We see a longer $\sigma$ leads to more fatalities, as well as delaying the efficacy of public health measures. However, while the effects of intervention are delayed, we note that, once visible, their impact occurs more suddenly. We observe non-physical behavior for $\sigma=20$, showing that the model may exhibit non-physical behaviors for larger $\sigma$. }\label{fig:PlotDeceased}
\end{figure}

\begin{figure}\centering
  \includegraphics[width=.7\textwidth]{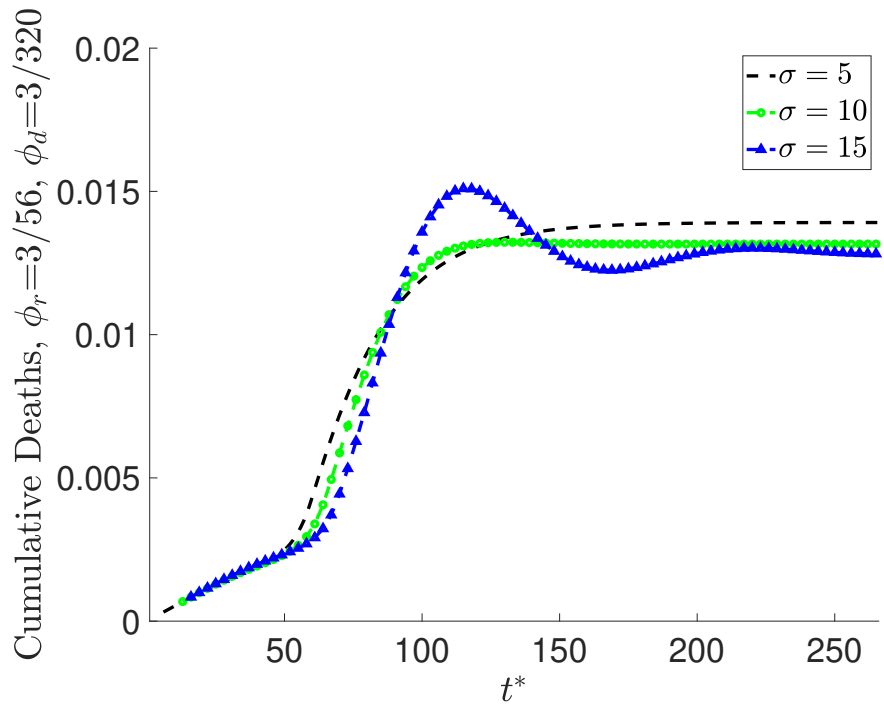} \caption{Total deceased for $\phi_r$=3/56, $\phi_d$=3/320. We see that with this choice of $\phi$ the equations become noticeably nonphysical for $\sigma=15$, with large decreases in the deceased compartment. However, we nonetheless observe stability for all $\sigma$. }\label{fig:Plot1480}
\end{figure}

\begin{figure}\centering
  \includegraphics[width=.7\textwidth]{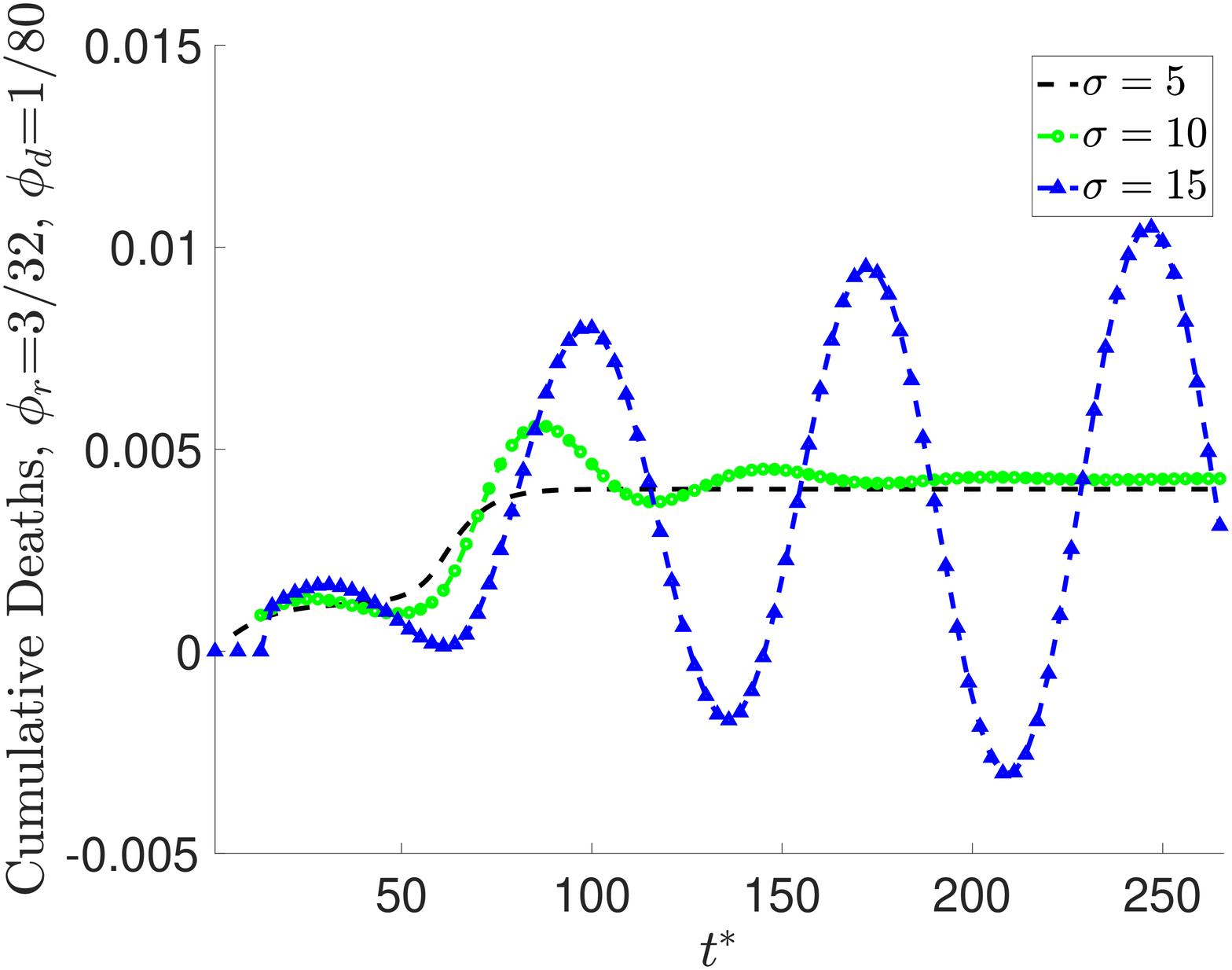} \caption{Total deceased for $\phi_r$=3/32, $\phi_d$=1/80. We see that with this choice of $\phi$ the equations are stable and physical for $\sigma=5$, nonphysical but stable for $\sigma=10$, and unstable for $\sigma=15$, as predicted by the stability condition.}\label{fig:Plot860}
\end{figure}

\subsubsection{Stability of the PDE Model}
We now examine the validity of the bound $(ii)$ in Theorem 3.1, derived for the ODE model, for the PDE model. As mentioned in the analysis section, we expect the results to hold identically in this case. In the preceding, we observe that for the parameters in Table \ref{tab:1DParametertTable}:
	$$\frac{1}{32}+\frac{3}{640} \approx .0359 < \frac{\pi}{2\sigma}$$
 for $\sigma$=5, 10, 15, 20, thus satisfying the condition for all considered $\sigma$. While we see see non-physical behaviors for both cases of $\sigma=20$, the solution does indeed remain stable. Such behaviors appear to be independent of both time-step size and time-integration scheme, and represent the model behavior. Thus, to guarantee physical behavior, stability is necessary but not sufficient.

\par For $\phi_r=3/56$ and $\phi_d=3/320$:
	$$\frac{3}{56}+\frac{3}{320} \approx .0629 < \frac{\pi}{2\sigma},$$

\noindent again satisfying the stability condition for all $\sigma$, we again see further numerical validation of the system stability. As shown in Fig. \ref{fig:Plot1480}, all values of $\sigma$ are stable. $\sigma=5$ and $\sigma=10$ both show physical behavior, avoiding oscillations and large decreases in the deceased compartment. While $\sigma=15$ is stable, the behavior is clearly nonphysical, and we see oscillations and decreases in the deceased compartment.
\par For $\phi_r=3/32$ and $\phi_d=1/80$:
	$$\frac{3}{32}+\frac{1}{80} \approx .1062 < \frac{\pi}{2\sigma}$$
for $\sigma=5,\,10$ but not for $\sigma=15$. The behavior of $\sigma=5$ is both physical and stable, and $\sigma=10$ is stable but nonphysical, as shown in Fig. \ref{fig:Plot860}. $\sigma=15$ violates the stability bound slightly, and we observe unstable behavior, with large increasing oscillations. This establishes not only the validity of the stability bound, but also its strictness. These results confirm the analysis, showing that the stability bound $(ii)$ in Theorem 3.1, holds for the PDE model. However, as the results also show clearly, stability does not imply a physical solution, as one still may observe negative values in the infected compartment and non-monotonic behavior in the recovered and deceased compartments. For sufficiently small values of $\sigma$ in comparison to $\phi_d$, $\phi_r$, we observe numerically that one may expect physical behavior in the solution. A positivity condition, a much stronger condition than stability, is required to guarantee physical solution behavior; this likely depends on the initial data and is an interesting direction for future work.
\subsubsection{Relationship with ODE solutions} In this case, we observe that the ODE provides a good approximation to the space-integrated PDE in 1D. This is expected from our analysis, as in this case we have taken $A=0$ and the variation in population density, while present, is not extreme. As mentioned previously, ODE models are obviously much less demanding than PDE models from the point of view of both implementation and simulation time \eli{as well as from analytical/stability analysis point of view}. For this reason, if possible, their use may be preferred over PDEs in certain contexts. These simulations confirm that, near equilibria, for small values of $A$, and relatively small variation in population density, the ODE may provide a surrogate for the PDE, if spatial differences are not considered important and one wishes to instead consider the population independently of space. \eli{On the other hand, the PDE model provides more reliable forecast when these factors are important, due to the presence of the spatial information.}

\subsection{Lombardy Simulation}
Our final numerical test models the outbreak of COVID-19 in the region of Lombardy, Italy. This test case was first used to validate a SEIRD model in \cite{VLABHPRYV2020}, in which the simulation results showed good agreement with the measured data. It was then examined in further detail in \cite{VLABHPRYV2020Due}, where other aspects of the problem, including the model sensitivity to diffusion, were \eli{investigated}. We will not discuss such aspects of this model here, as the focus of the present work is on the delay differential equation model. Hence, for a more detailed discussion of the aforementioned results, we kindly refer the reader to \cite{VLABHPRYV2020, VLABHPRYV2020Due}.  
\par For the spatial discretization, we utilize an unstructured triangular mesh with 41625 elements. For the temporal discretization, we use the Backward-Euler method with a time step $\Delta t=.25$ days. We solve the nonlinear problem at each time step using a Picard-style linearization, and the corresponding linear systems are solved using the GMRES algorithm with a Jacobi-style preconditioner. At all boundaries, we assign no-flux boundary conditions, corresponding to total isolation. The parameter values are reported in Table \ref{tab:2DParametertTable}. We note that these differ from those shown in \cite{VLABHPRYV2020Due, VLABHPRYV2020}; this is primarily due to differences resulting from the fact that, in the current delay model, asymptomatic individuals are now considered within the `infected' compartment.

\begin{table}
\begin{center}

\resizebox{\textwidth}{!}{%
\begin{tabular}{ |c|c|c|c|c|c|c|c| } 
\hline
Parameter   &  Units  & Feb.27-Mar.9 & Mar.9-22 & Mar.22-28 &Mar.28-May3 &May3- \\
\hline\hline
$\sigma$  & Days  & 8  & 8 & 8 &8 & 8 \\ \hline
$\beta_e$ & Persons$^{-1}\cdot$Days$^{-1}$ & 3.75$\cdot 10^{-4}$ &  3.11$\cdot 10^{-5}$ &  2.0625$\cdot 10^{-5}$ &  1.5$\cdot 10^{-5}$ & 2.75$\cdot 10^{-5}$ \\ \hline
$\beta_i$ & Persons$^{-1}\cdot$Days$^{-1}$ & 3.75$\cdot 10^{-4}$ &  3.11$\cdot 10^{-5}$ &  2.0625$\cdot 10^{-5}$ &  1.5$\cdot 10^{-5}$ & 2.75$\cdot 10^{-5}$ \\ \hline
$\phi_r$  & Days$^{-1}$  & 3/64  & 3/64 & 3/64 &3/64 & 3/64 \\ \hline
$\phi_d$  & Days$^{-1}$  & 3/320  & 3/320 & 3/320 &3/320 & 3/320 \\ \hline
$\overline\nu_s $  & km$^{2}\cdot$ Persons$^{-1}\cdot$Days$^{-1}$   & 4.35$\cdot 10^{-2}$ & 1.98$\cdot 10^{-2}$ & 0.9$\cdot 10^{-2}$ & 0.75$\cdot 10^{-2}$ & 2.175$\cdot 10^{-2}$ \\ \hline
$\overline\nu_i $  & km$^{2}\cdot$ Persons$^{-1}\cdot$Days$^{-1}$  & 2.175$\cdot 10^{-2}$ & 1.$\cdot 10^{-2}$ & 0.45$\cdot 10^{-2}$ & 0.325$\cdot 10^{-2}$ & 1.0625$\cdot 10^{-2}$ \\ \hline$\overline\nu_r $  & km$^{2}\cdot$ Persons$^{-1}\cdot$Days$^{-1}$  & 4.35$\cdot 10^{-2}$ & 1.98$\cdot 10^{-2}$ & 0.9$\cdot 10^{-2}$ & 0.75$\cdot 10^{-2}$ & 2.175$\cdot 10^{-2}$ \\ \hline
$\overline{A} $  & Persons  & 1.0$\cdot 10^{3}$ & 1.0$\cdot 10^{3}$ & 1.0$\cdot 10^{3}$ &1.0$\cdot 10^{3}$ & 1.0$\cdot 10^{3}$ \\ \hline

%
%

\end{tabular}}
\caption{Parameter values for the 2D Lombardy simulations. The values change with date as these correspond to various restrictions (or relaxtions) taken by the government during the epidemic. We note that these parameters are not normalized in space.}
  \label{tab:2DParametertTable} 
\end{center}
\end{table}

\begin{figure}\centering
  \includegraphics[width=.7\textwidth]{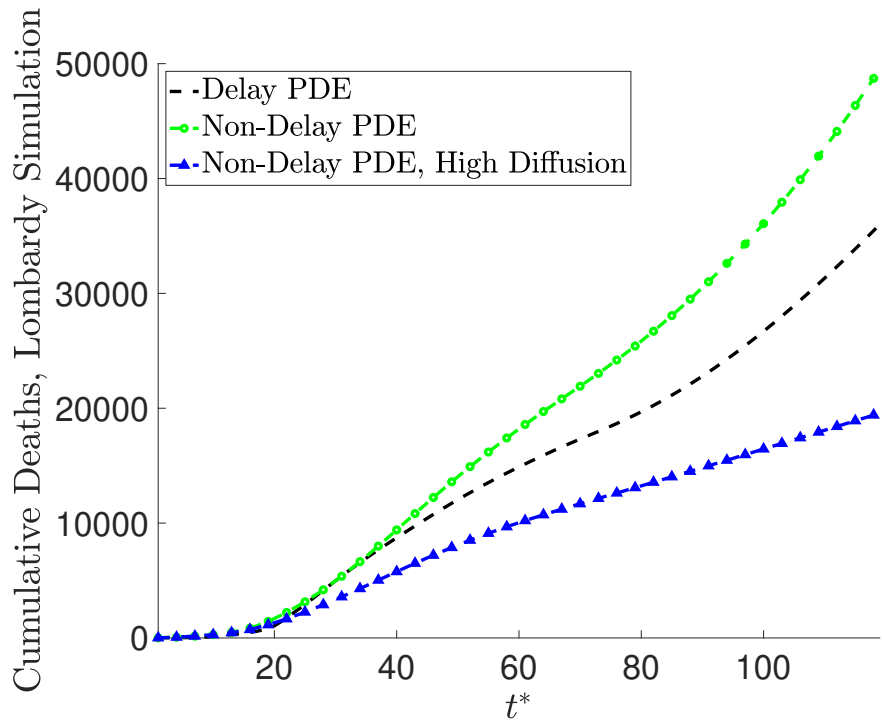} \caption{Deceased individuals for the Lombardy test case. We show the delay PDE results together with two other non-delay model formulations, validated against measured data. We see that the results exhibit similar qualitative behavior, establishing the potential of the delay PDE model to produce realistic simulation results in time and space.}\label{fig:PlotLomb}
\end{figure}

\begin{figure}\centering
  \includegraphics[width=.8\textwidth]{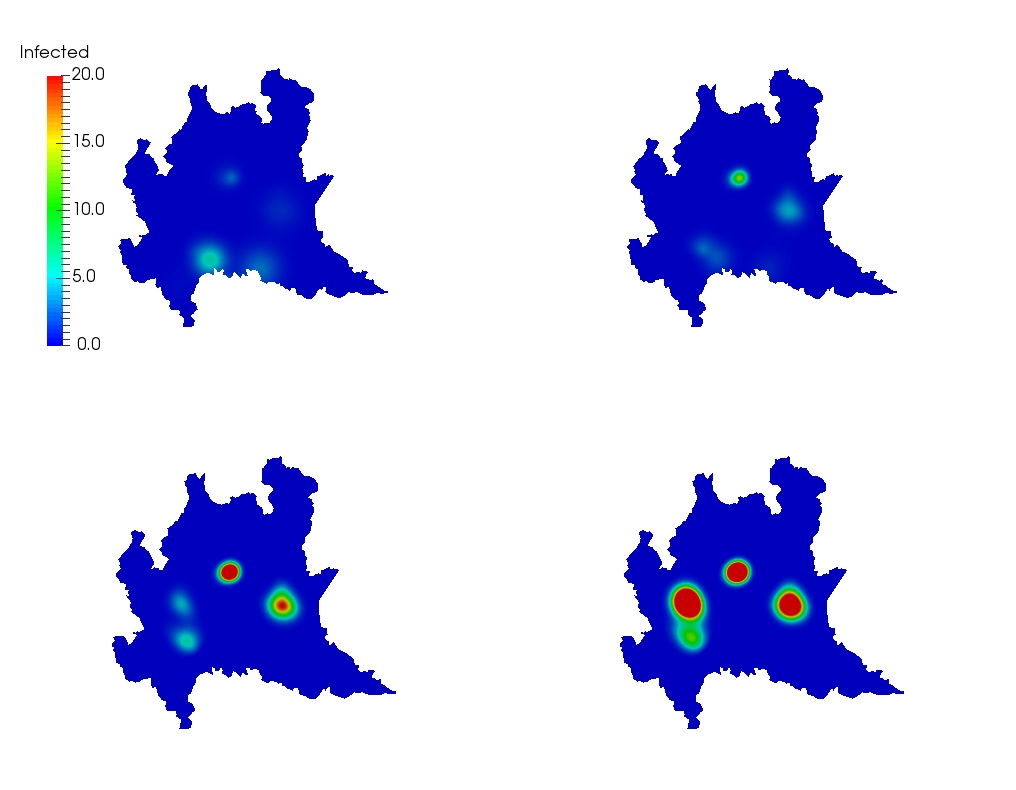} \caption{Evolution of the epidemic in Lombardy using the Delay PDE. Clockwise, from top left: Day 1, 5, 10, 20. The outbreak begins with a clusters of cases in the south of the region, before moving north into the region's large cities.}\label{fig:EpidemicEvolution}
\end{figure}

\par In Fig. \ref{fig:PlotLomb}, we show how the results compare to two formulations of the problem shown in \cite{VLABHPRYV2020Due}: the `baseline' case, which features the same diffusive parameters used here, and the case with doubled diffusion. Note that we focus here on the deceased compartment, rather than the infectious compartment, due to the differing definitions `infected' between the two models. We see similar qualitative behavior to the baseline case, with an $R^2$ correlation coefficient between the of 99.8\%. Over the first ~40 days, the behavior is nearly identical, with the results beginning to differ somewhat further in time. The agreement with the baseline case, rather than the high-diffusion case from \cite{VLABHPRYV2020Due}, suggests that the delay model does not interfere with the diffusive behavior generally.  
\par In Fig. \ref{fig:EpidemicEvolution}, we observe the evolution of the epidemic in space. Starting from the top-left and moving clockwise, we show the density of infected individuals on days 1, 5, 10, and 20. What begins as a small cluster of cases in the south of the region moves northward, into the region's large cities (Milan, Bergamo, and Brescia). While the initially affected regions improve rapidly, in the large cities the epidemic continues to grow. This is consistent with the observed data, and with the simulations using non-delay models shown in \cite{VLABHPRYV2020}.
\par In order to assess the stability bound, we performed a second numerical simulation. We note that, strictly speaking, as $A$ is nonzero in this case, the assumptions of the stability bound do not strictly hold. However, in other simulations (not shown), the qualitative behavior of the model with $A=0$ was similar to that shown here, and so we may expect the bound to hold heuristically. To better observe the stability behavior, we set $\beta_i$, $\beta_e = 3.75\cdot 10^{-4}$ from Feb. 27-Mar. 4, and $\beta_i$, $\beta_e = 3.11\cdot 10^{-5}$ for the remainder of the simulation. We then let\footnote{The remainder of parameter values are the same as in Table \ref{tab:2DParametertTable}} $\phi_r=1/8$, $\phi_d=1/80$, $\sigma=12$. Hence:
$$ 1/8 + 1/80 =.1375 > \frac{\pi}{2*12}.$$
In accordance with the theory, we do not expect stability for this case. Indeed, this is what we observe. In Fig. \ref{fig:PlotStabLomb}, we plot the total number of active infections for the stable (previously discussed) case, as well as the unstable case. While in the stable case we observe the expected behavior, in which infections remain positive and grow or decrease in response to the pandemic-arresting measures in place, we instead observe large oscillations between positive and negative values for the unstable case. We show this behavior in time and space in Fig. \ref{fig:DelayOscillations}, where one sees the oscillations concentrated in the heavily-affected outbreak zones within the region. The frequency of these oscillations appears to be nonuniform, and dependent on infection concentration. Whether this provides any important information is, unclear, though it is an interesting observation and perhaps worthy of some investigation.

\begin{figure}\centering
  \includegraphics[width=.7\textwidth]{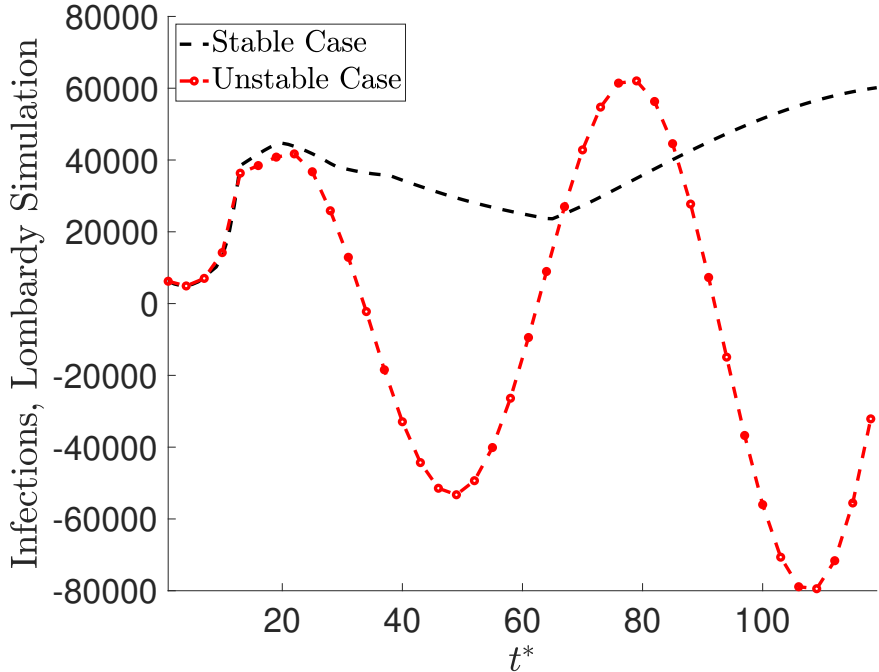} \caption{Active infections for the Lombardy simulations in both the stable and unstable regimes. In the stable regime, we see the expected behavior; infections remain positive and change in response to pandemic-arresting measures. In the unstable regime, we instead see spurious nonphysical oscillations between positive and negative.}\label{fig:PlotStabLomb}
\end{figure}

\begin{figure}\centering
  \includegraphics[width=.8\textwidth]{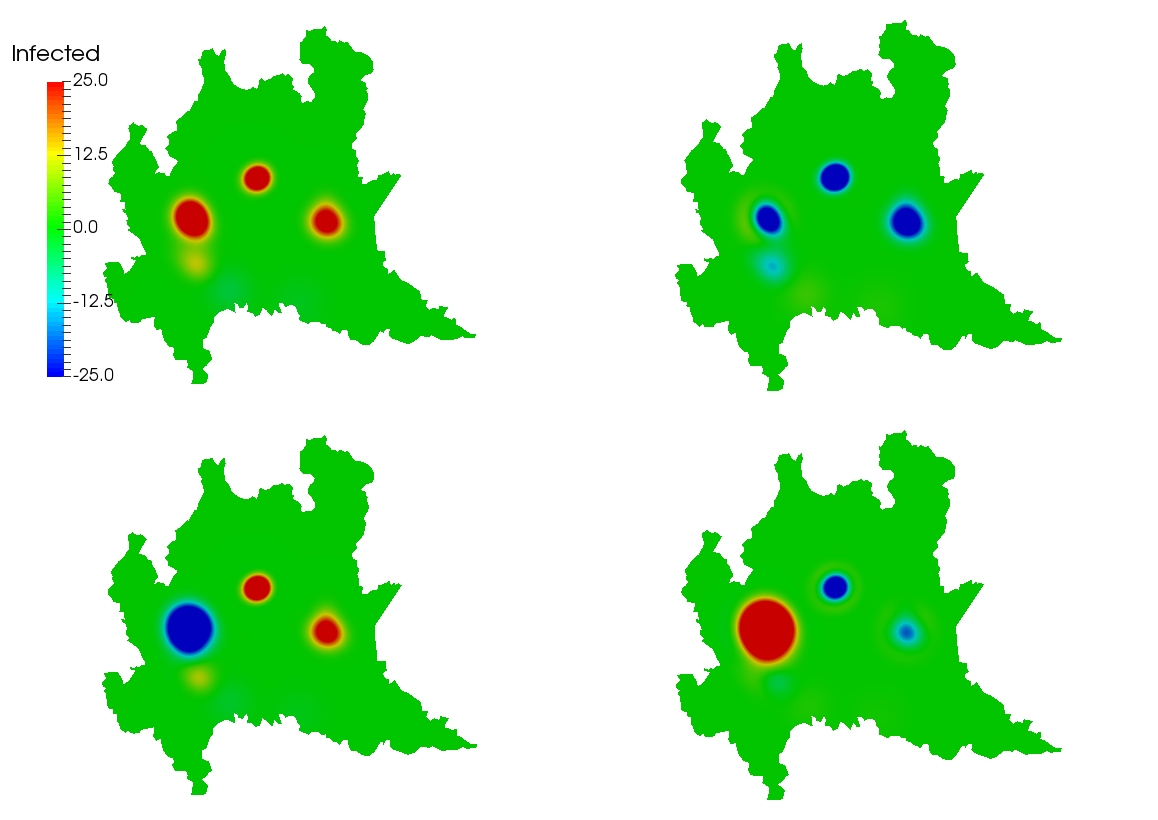} \caption{The simulation of Lombardy in the unstable regime, clockwise from left: days 20, 40, 60, 80. We see interesting oscillatory behavior between positive and negative throughout the region and the primary infectious zones.}\label{fig:DelayOscillations}
\end{figure}

\section{Conclusions}
We have presented a new formulation for epidemic models utilizing delay differential equations in both an ODE and PDE formulation. We have established stability results for the ODE formulation, which were then confirmed with numerical experiments. For the PDE model, we observed interesting dynamics regarding the relationship between the delay time and lockdowns, as well as contagion in general. We further showed, with numerical evidence, that the stability bounds established for the ODE also hold for the PDE, and that in some situations the ODE may a reasonable surrogate for the PDE. We then concluded with a simulation on a realistic problem, in which we showed that the delay PDE can reproduce reality at a reasonable level, by obtaining results similar to non-delay models shown in other work. We also analyzed the spatiotemporal behavior of the unstable regime, finding it produced large oscillatory behavior between positive and negative values within the heavily impacted regions. 
\par There are many worthwhile directions for future work on this model, and the area generally. While we provided theoretical and numerical evidence of stability, we also showed for both the PDE and ODE that stability does not necessarily guarantee physical solution behavior. A stronger result, such as a positivity condition, is needed to provide such a result. Lastly, we have restricted ourselves to the constant delay case as a first step, but the most general and realistic models of this type should incorporate state-dependent delays. Although such models have many theoretical and numerical difficulties, many epidemics and other related phenomena exhibit this type of behavior \cite{MurrayI, MurrayII, BZ03, DGVW95}.


	 \bibliographystyle{plain}
\bibliography{covid19_v1.bib}

\end{document}